\documentclass[10pt]{article}

\usepackage{amsmath}
\usepackage{amsthm}
\usepackage{amssymb}
\usepackage{amsfonts}
\usepackage{graphicx}
\usepackage{verbatim}
\usepackage{mathrsfs}
\usepackage{subfigure}
\usepackage{fancyhdr}
\usepackage{latexsym}
\usepackage{graphicx}
\usepackage{dsfont}
\usepackage{color}
\usepackage[hidelinks]{hyperref}
\usepackage{bm}
\usepackage{drawmatrix}
\usepackage{enumitem}
\usepackage[percent]{overpic}

\newcommand{\R}{\mathbb R}
\newcommand{\N}{\mathbb N}
\newcommand{\Z}{\mathbb Z}

\newcommand{\ball}{\mathcal B}

\newcommand{\X}{\bm{\mathcal X}}

\newcommand{\XX}{\mathcal X}
\newcommand{\YY}{\mathcal Y}

\newcommand{\CC}{\mathcal C}
\newcommand{\p}{\bm{p}}
\newcommand{\f}{\bm{f}}
\newcommand{\F}{\bm{F}}
\newcommand{\g}{\bm{g}}
\newcommand{\T}{\bm{T}}
\newcommand{\0}{\bm{0}}

\newcommand{\M}{\bm{M}}
\newcommand{\A}{\bm{A}}
\newcommand{\B}{\bm{B}}
\newcommand{\bu}{\bm{u}}
\newcommand{\bv}{\bm{v}}
\newcommand{\bC}{\bm{C}}

\newcommand{\barX}{\bar{X}}
\newcommand{\br}{\bar r}
\newcommand{\brho}{\bar\rho}
\newcommand{\bX}{\bm{X}}

\newcommand{\cA}{{\mathcal A}}  
\newcommand{\cB}{{\mathcal B}}  
\newcommand{\cF}{{\mathcal F}}  

\def\txtd{{\textnormal{d}}}

\def\txtD{{\textnormal{D}}}

\def\ra{\rightarrow}
\def\I{\infty}

\newcommand{\fC}{\mathfrak{C}}
\definecolor{Blue}{RGB}{55,65,180}
\definecolor{Red}{RGB}{200,0,0}
\definecolor{Green}{RGB}{0,120,0}

\theoremstyle{plain}
\newtheorem{theorem}{Theorem}[section]
\newtheorem{proposition}[theorem]{Proposition}
\newtheorem{lemma}[theorem]{Lemma}
\newtheorem{corollary}[theorem]{Corollary}

\newtheorem{remark}[theorem]{Remark}
\newtheorem{definition}[theorem]{Definition}

\allowdisplaybreaks

\DeclareFontFamily{U}{matha}{\hyphenchar\font45}
\DeclareFontShape{U}{matha}{m}{n}{
      <5> <6> <7> <8> <9> <10> gen * matha
      <10.95> matha10 <12> <14.4> <17.28> <20.74> <24.88> matha12
      }{}
\DeclareSymbolFont{matha}{U}{matha}{m}{n}
\DeclareFontSubstitution{U}{matha}{m}{n}
\DeclareFontFamily{U}{mathx}{\hyphenchar\font45}
\DeclareFontShape{U}{mathx}{m}{n}{
      <5> <6> <7> <8> <9> <10>
      <10.95> <12> <14.4> <17.28> <20.74> <24.88>
      mathx10
      }{}
\DeclareSymbolFont{mathx}{U}{mathx}{m}{n}
\DeclareFontSubstitution{U}{mathx}{m}{n}
\DeclareMathDelimiter{\VERT}{0}{matha}{"7E}{mathx}{"17}

\title{Rigorous Validation of Stochastic Transition Paths}

\author{Maxime Breden \thanks{Technical University of Munich, 
Faculty of Mathematics, Research Unit ``Multiscale and Stochastic 
Dynamics", 85748 Garching b. M\"unchen, Germany. \texttt{maxime.breden@tum.de}} 
\and Christian Kuehn \thanks{Technical University of Munich, Faculty 
of Mathematics, Research Unit ``Multiscale and Stochastic Dynamics", 
85748 Garching b. M\"unchen, Germany. \texttt{ckuehn@ma.tum.de}}}

\begin{document}

\maketitle

\begin{abstract}
Global dynamics in nonlinear stochastic systems is often difficult to analyze 
rigorously. Yet, many excellent numerical methods exist to approximate these
systems. In this work, we propose a method to bridge the gap between computation
and analysis by introducing rigorous validated computations for stochastic 
systems. The first step is to use analytic methods to reduce the stochastic
problem to one solvable by a deterministic algorithm and to numerically compute
a solution. Then one uses fixed-point arguments, including a combination of 
analytical and validated numerical estimates, to prove that the computed solution
has a true solution in a suitable neighbourhood. We demonstrate our approach by
computing minimum-energy transition paths via invariant manifolds and heteroclinic
connections. We illustrate our method in the context of the classical M\"uller-Brown 
test potential.
\end{abstract}

\begin{center}
\begin{small}
\textbf{Keywords:} validated computation, rigorous numerics, 
stochastic dynamics, fixed-point problem,
minimum-energy path, heteroclinic orbit, invariant manifolds.
\end{small}
\end{center}

\tableofcontents

\section{Introduction}

\subsection{Validated computations for stochastic systems: general strategy}

Steady states, connecting orbits and more general invariant sets are 
key objects to understand the long time behavior of a deterministic 
dynamical system. There are many theoretical existence results, but 
these results are often non-quantitative or difficult to apply, 
specially when the system is highly nonlinear. These limitations 
motivated the development of a field now called rigorous numerics. Its
goal is to use computer-assisted verification in combination with 
analytic estimates to prove theorems. For example, it has turned out
that the existence of many chaotic attractors could only be established 
with computer-assisted methods~\cite{Haiduc1,Tucker}. For detailed 
reviews about the history and applications of rigorous numerics, we refer 
to~\cite{MirMis15,Rum10,Tuc11,BerLes15} and the references therein. 

Understanding the global behavior of a stochastic nonlinear dynamical system 
is even harder. The difficulties dealing with global dynamics induced by 
the deterministic nonlinear part remain, and the noise usually adds another 
layer of technical challenges. Nevertheless, it is frequently possible to 
reduce particular questions to the study of associated deterministic systems. 
Some examples include action functionals in large deviation 
theory~\cite{FreWen13,Varadhan}, variance estimates for sample 
paths~\cite{BerglundGentz}, pathwise descriptions of stochastic 
equations~\cite{Lyons}, moment closure methods~\cite{KuehnMC} 
and partial differential equation (PDE) methods 
for stochastic stability~\cite{ArnoldMarkowichToscaniUnterreiter,Khasminskii1}. 
However, if the original stochastic system has nonlinear terms, an associated 
deterministic reduced problem one will usually also contain nonlinear terms.

In this work, we propose to combine the two strategies outlined above, i.e., to 
first use deterministic reduction and then employ rigorous a-posteriori validation 
to obtain information about the global behavior of stochastic dynamical systems.

\subsection{Validated computations for stochastic systems: possible applications}

To illustrate the feasibility of our approach, we study a particular example. Here we only outline our strategy and we refer for more technical background to Section~\ref{sec:background}.

Consider a metastable stochastic gradient system~\cite{BovierdenHollander} 
with potential $V$ modelled by a stochastic ordinary differential equation 
(SODE) with additive noise. The minima of $V$ correspond to deterministically
stable, and for small noise stochastically metastable, states. Noise-induced 
transitions between minima are known to occur most frequently along 
minimum energy paths (MEPs) of a deterministic action functional. Many 
efficient numerical methods have been developed, particularly in the context
of applications in chemistry, to compute 
MEPs~\cite{ERenVanden-Eijnden,HenUbeJon00,SheppardTerellHenkelman}. 
For our example, MEPs can also be characterized as concatenations of heteroclinic 
connections~\cite{DoedelFriedman,GH} between the minima and intermediate saddle points. 
The main steps that we use in this paper to rigorously validate a MEP are: (I) locate 
and validate saddles and minima, (II) introduce an equivalent polynomial 
vector field, and (III) validate heteroclinic orbits for this polynomial vector 
field. Step (III) is more involved and can be decomposed in three sub-steps:
(IIIa) compute and validate the local unstable manifold of the saddles, (IIIb) 
validate trapping regions around the minima, and (IIIc) compute and validate 
orbits connecting unstable manifolds and trapping regions. The output of the
steps (I)-(III) is a theorem providing the existence and detailed quantitative
location of the MEP. We use the standard M\"uller-Brown potential as a test
case to present these steps.

\begin{remark}
We aim at making the paper as accessible as possible to a broad audience, therefore we choose to focus on using the steps (I)-(III) for the M\"uller-Brown potential, which limits the amount of technicalities required, but emphasize that our framework is of course not limited to this specific example. We also mentioned that the framework itself could be broadened, by using different rigorous numerics techniques than the one presented in full details here (see Section~\ref{sec:background_validation}), and by tailoring the steps (I)-(III) (and in particular step (II)) to the problem at hand. We expand on these possible generalizations of our work in Section~\ref{sec:generalizations}.
\end{remark} 

While we focus on the case of MEPs in this paper, our \emph{general strategy} of reducing a stochastic system to a deterministic one and then studying this (nonlinear) deterministic system using rigorous numerics can find applications in many other stochastic contexts. For instance, the mean first passage time of an SODE to a prescribed boundary satisfies a PDE~\cite{FreWen13}, which could be studied using rigorous numerics. Another example is given by recently developed methods, which 
numerically calculate the  local fluctuations in stochastic systems 
around deterministic equilibria by using matrix 
equations~\cite{KuehnSDEcont1,KuehnSPDEcont}, and for which rigorous numerics would be directly applicable. In a different direction, we mention the recent work~\cite{GalMonNis17}, where the existence of noise induced order is established, using computer-assisted techniques based on transfer operators.

\subsection{Main contributions and outline}
Our main contributions in this work are the 
following:

\begin{itemize}
 \item We establish the first link between stochastic dynamics and rigorous a posteriori validation techniques.
 \item We apply our methodology to the case of MEPs in metastable stochastic
system.
 \item We demonstrate the feasibility using the standard test case
of the M\"uller-Brown potential.
 \item We also improve validation methods themselves by developing a priori 
almost optimal weights for the function spaces employed (see 
e.g.~Proposition~\ref{prop:frob}).
\end{itemize}

The paper is structured as follows: In Section~\ref{sec:background} we collect
different background from large deviation theory, from dynamical systems, from 
available algorithmic approaches to compute MEPs, and from rigorous numerics.
Although these techniques are certainly well-known within certain communities,
it seems useful to us to briefly introduce them here as we establish new links
between several fields. In Section~\ref{sec:main_steps} we provide the basic
setting for the rigorous validation to be carried out. In 
Sections~\ref{sec:manifold}-\ref{sec:connection} we carry out the technical
estimates and required computations for step (III) regarding the local invariant 
manifolds and connecting orbits. We conclude in Section~\ref{sec:conclusions} by presenting the output of the whole procedure for the M\"uller-Brown potential, and by mentioning possible extensions of this procedure. The required code for the validated numerics
can be found at~\cite{BreKue18_code}.   

\section{Background}
\label{sec:background}

\subsection{Large Deviation Theory and Minimum Energy Paths}
\label{sec:background_stochastics}

Consider a smooth potential $V:\R^n \to\R$ with two local minima $m_{1*}$ 
and $m_{2*}$, separated by a saddle $s_*$. The gradient flow ordinary 
differential equation (ODE) induced by the potential is 
\begin{equation}
\label{eq:ODE}
x'=\frac{\txtd x}{\txtd t}=-\nabla V(x),\qquad x=x(t)\in\R^n,~x(0)=x_0. 
\end{equation}
Observe that $m_{1*}$ and $m_{2*}$ are locally asymptotically stable equilibria for~\eqref{eq:ODE} while $s_*$ is unstable. Since the Jacobian $\txtD (-\nabla V)$ of the vector field is the (negative) Hessian of $V$, it follows that $-\txtD^2 V$ has only real eigenvalues. For the exposition here, we assume that $n\geq 2$ and $s_*$ is a saddle point with a single unstable direction,
i.e., $-\txtD^2V(s_*)$ has a single positive eigenvalue. Denote by $\cB_1$ and $\cB_2$ the basins of attraction of $m_{1*}$ and $m_{2*}$. If $x_0\in \cB_1$ (resp.~$x_0\in\cB_2$) then $x(t)$ converges to $m_{1*}$ (resp.~to $m_{2*}$) as 
$t\ra +\I$. In particular there is no solution that connects $m_{1*}$ and 
$m_{2*}$. For systems under the influence of noise, the situation is
different. Let $(\Omega,\cF,\cF_t,\mathbb{P})$ be a filtered probability space 
and consider a vector $W(t)=W=(W_1,W_2,\ldots,W_n)^\top$ of independent
identically distributed (iid) Brownian motions. Introducing additive noise 
in~\eqref{eq:ODE} gives the SODE
\begin{equation}
\label{eq:SDE}
\txtd x=-\nabla V(x)~\txtd t +\varepsilon~\txtd W,
\end{equation}
where $\varepsilon$ is a small parameter ($0<\varepsilon\ll 1$). \emph{Noise-induced
transitions} occur for~\eqref{eq:SDE}, e.g., fixing small neighbourhoods of
$m_{1*}$ and $m_{2*}$, there exist sample paths $\gamma=\gamma(t)$ 
solving~\eqref{eq:SDE} meeting both neighbourhoods in finite time. For the 
gradient system we considered here, it is well-known that the most 
probable path $\gamma$, also called \emph{minimum energy path (MEP)}, passes 
through the saddle point $s_*$. Also the 
mean-first passage time is well-understood via the classical 
Arrhenius-Eyring-Kramers law~\cite{Arrhenius,Berglund3,BovEckGayKle04,BovGayKle05,Eyr35,Kramers,LelNie15},
which states that transition probabilities are exponentially small in
the noise level $\varepsilon$ as $\varepsilon\ra 0$. A general framework
for the small-noise case is provided by \emph{large deviation 
theory}~\cite{FreWen13,Varadhan}, which can be formulated for more general
SODEs such as 
\begin{equation}
\label{eq:SDE1}
\txtd X=f(X)~\txtd t +\varepsilon~\txtd W,
\end{equation}
where $f$ is a sufficiently smooth vector field. Let $\phi:[0,T]\ra \R^n$
be an absolutely continuous path over fixed time $T>0$ and define the 
\emph{action} (or \emph{action functional})
\begin{equation}
\label{eq:FW}
S_T(\phi)=\frac{1}{2}\int_0^T L(\phi,\phi')~\txtd t,
\end{equation}
where the Lagrangian $L$ is given by
\begin{equation}
\label{eq:FW1}
L(\phi,\phi')=\|\phi'-f(\phi)\|_2^2.
\end{equation}
We can think of $S_T(\phi)$ as the energy spent following a solution $\phi$
against the deterministic ($\varepsilon=0$) flow for~\eqref{eq:SDE1}. Then
Wentzell-Freidlin theory~\cite{FreWen13} states that the probability of
a solution $X=X(t)$ to pass through a $\delta$-tube around 
$\phi(t)$ for $t\in[0,T]$ is given by
\begin{equation*}
\mathbb{P}\left(\sup_{t\in[0,T]}\|X-\phi\|_2<\delta\right)\approx 
\exp\left(-\frac{1}{\varepsilon^2} S_T(\phi)\right)
\end{equation*}
for sufficiently small $\varepsilon$ and $\delta$;~cf.\cite[Sec.~3, Thm.~2.3]{FreWen13}. More generally, one can just take any set $\cA\subset \cF$ of random
events and obtain the large deviation principle
\begin{equation*}
\lim_{\varepsilon\ra 0}\varepsilon^2\ln\left(\mathbb{P}\left(X\in\cA\right)\right)= 
-\inf_{\phi\in\cA}S_T(\phi).
\end{equation*}
Note carefully that this formulation has converted a stochastic problem
of calculating/estimating a probability into a deterministic optimization
problem. 

If we want to specify the transition problem between
two points $x_0$ and $x_1$, we should set
\begin{equation*}
\cA=\{X(0)=x_0,\,X(T)=x_1\},
\end{equation*}
and consider
\begin{equation}
\label{eq:minimization}
\inf_{\substack{\phi\in\cA \\ T>0 }}S_T(\phi).
\end{equation}
If a minimizer $\phi_*$ of the action functional exists, one refers to $\phi_*$ as a minimum action path (MAP). We reserve MEP for 
the case of a minimizer in gradient systems for a transition between states. 

In practice it is often more important to know the geometrical path described by $\phi_*$ (i.e. the set $\{\phi_*(t)\,|\, t\in[0,T]\}$) rather than the function $\phi_*$ itself. One can then consider the \emph{geometric action functional}~\cite{HeyVan08bis}
\begin{equation}
\label{eq:GAF}
\tilde{S}_T(\phi)=\int_0^T \Big(\left\Vert \phi' \right\Vert_2 \left\Vert f(\phi) \right\Vert_2 - \langle \phi',f(\phi) \rangle\Big) \txtd t,
\end{equation}
which is only sensitive to the path described by $\phi$ but not to the function $\phi$ itself. More precisely, for any time reparametrization $\sigma:[0,T]\to[0,\tau]$ preserving orientation, one has $\tilde{S}_T(\phi) = \tilde{S}_\tau(\phi\circ\sigma)$. Besides, the geometric action functional can also be used to characterize MAPs/MEPs since
\begin{equation}
\label{eq:eq_inf}
\inf_{\substack{\phi\in\cA \\ T>0 }}S_T(\phi) = \inf_{\substack{\phi\in\cA \\ T>0 }} \tilde S_T(\phi) = \inf_{\phi\in\cA} \tilde S_1(\phi).
\end{equation}
Indeed, first notice that, using 
\begin{equation}
\label{eq:ab}
\frac{1}{2}\left(\left\Vert \phi' \right\Vert_2^2 + \left\Vert f(\phi) \right\Vert_2^2\right) \geq \left\Vert \phi' \right\Vert_2 \left\Vert f(\phi) \right\Vert_2,
\end{equation}
we get $\tilde S_T(\phi) \leq S_T(\phi)$, hence 
\begin{equation}
\label{eq:ineq_inf}
\inf_{\substack{\phi\in\cA \\ T>0 }} \tilde S_T(\phi) \leq  \inf_{\substack{\phi\in\cA \\ T>0 }} S_T(\phi).
\end{equation}
By considering a reparametrization $\sigma$ such that $\left\Vert (\phi\circ\sigma)' \right\Vert_2 = \left\Vert f(\phi\circ\sigma) \right\Vert_2$ we have that~\eqref{eq:ab} is an equality for $\phi\circ\sigma$, therefore $\tilde S_T(\phi) = \tilde S_\tau(\phi\circ\sigma) = S_T(\phi\circ\sigma)$ which together with~\eqref{eq:ineq_inf} yields the first equality in~\eqref{eq:eq_inf}. The second equality is a direct consequence of the invariance of $\tilde S$ by time reparametrization.

For gradient systems, the geometric action functional also provides another useful characterization of MEPs. Consider for instance the transition between $m_{1*}$ and $s_*$ for~\eqref{eq:SDE}, i.e. set
\begin{equation*}
\cA=\{X(0)=m_{1*},\,X(T)=s_*\}.
\end{equation*}
We then have
\begin{align*}
\tilde{S}_T(\phi) & =\int_0^T \Big(\left\Vert \phi' \right\Vert_2 \left\Vert \nabla V(\phi) \right\Vert_2 + \langle \phi',\nabla V(\phi) \rangle\Big) \txtd t \\
&\geq 2\int_0^T \langle \phi',\nabla V(\phi) \rangle \txtd t \\
&= 2\left(V(s_*)-V(m_{1*})\right),
\end{align*}
where the inequality becomes an equality if $\phi'(t)$ and $\nabla V(\phi(t))$ are positively collinear for all $t\in[0,T]$.

It is
usually not possible to calculate MAPs/MEPs analytically due to the
need to solve a nonlinear deterministic problem. However, there are many
numerical methods, mainly motivated by chemical applications, and mostly based on the geometric action potential. For MEPs 
in gradient systems, there is the classical nudged elastic band 
method~\cite{HenUbeJon00} as well as the string 
method~\cite{WeiRenVan07}. Both methods start with an initial 
piecewise-linear path in the phase space $\R^n$. Making use of the above characterization, the path is then evolved so that $\phi'(t)$ and $\nabla V(\phi(t))$ align. Both methods show often quite similar performance 
as just the re-meshing strategies differ~\cite{SheppardTerellHenkelman}. In cases when the final state $m_{2*}$ is not known a priori, the dimer method~\cite{HenJon99} can be used. Of course, one can also directly discretize the action and try to
minimize it, which leads to another natural set of numerical 
methods~\cite{ERenVanden-Eijnden3}. Another option is to look at the Euler-Lagrange equation associated to the geometrical action path~\cite{HeyVan08}. Finally, one can also numerically
compute MAPs/MEPs by noticing that in many situations they are (concatenations 
of) orbits for the deterministic system for $\varepsilon=0$. For 
example, consider again the gradient system ODE~\eqref{eq:ODE}  
with the three critical points $m_{1*}$, $s_*$, $m_{2*}$, and assume we are interested in the transition going from $m_{1*}$ to $m_{2*}$ via $s_*$. Then the path associated to
\begin{equation}
\phi_*=\{m_{1*}\}\cup \gamma_{s_*m_{1*}}\cup \{s_*\}\cup \gamma_{s_*m_{2*}}\cup \{m_{2*}\},
\end{equation}  
where $\gamma_{s_*m_{1*}}$ is a heteroclinic orbit from $s_*$ to $m_{1*}$  
and $\gamma_{s_*m_{2*}}$ is a heteroclinic orbit from $s_*$ to $m_{2*}$, i.e.,
\begin{equation}
\lim_{t\ra -\I}\gamma_{s_*m_{1*}}(t)=s_*,\quad 
\lim_{t\ra +\I}\gamma_{s_*m_{1*}}(t)=m_{1*},\quad 
\lim_{t\ra -\I}\gamma_{s_*m_{2*}}(t)=s_*,\quad 
\lim_{t\ra +\I}\gamma_{s_*m_{2*}}(t)=m_{2*},
\end{equation}
is a MEP. Of course, the heteroclinic solution going from $s_*$ to $m_{1*}$ does not truly describe the transition, as it goes in the wrong direction. However the function $\tilde \gamma_{s_*m_{1*}}$ defined by $\tilde \gamma_{s_*m_{1*}}(t)=\gamma_{s_*m_{1*}}(-t)$ does, as it satisfies $\tilde \gamma_{s_*m_{1*}}'= \nabla V(\tilde \gamma_{s_*m_{1*}})$ and therefore minimizes the geometric action. Since the path associated to $\tilde \gamma_{s_*m_{1*}}$ and $\gamma_{s_*m_{1*}}$ is the same, the path associated to $\phi_*$ is indeed a MEP.
Hence,
a MEP algorithm can also be built around calculating saddles and heteroclinic
connections, which is the strategy we pursue here using rigorous numerics, i.e, 
we turn a numerical computation into a rigorous proof including error estimates
for the location of the MEP.

\subsection{Rigorous A-Posteriori Validation Methods}
\label{sec:background_validation}

First, let us briefly mention that there are roughly two types of rigorous numerics techniques. The first kind are often described as \emph{geometric} or \emph{topological} methods. They aim at controlling the numerical solution directly in phase space, and are based on shadowing techniques~\cite{KocPalCoo07}, covering relations~\cite{ZglGid04}, cone conditions~\cite{Zgl09} and rigorous integration of the flow (see for instance~\cite{BerMak98,Zgl08}). The second kind are referred to as \emph{functional analytic} methods and aim at validating a posteriori the numerical solution by a fixed point argument (see~\cite{Ces63,Ura65} for some seminal works in this direction, and~\cite{DayLesMis07,Nak01,Plu01} for some more recent references). 

In this paper, we are going to use a functional analytic approach to validate numerically obtained MEPs by rigorously computing a chain of heteroclinic orbits. Let us mention that this is far from being the first time that rigorous numerics are used to study connecting orbits, and that there is a rapidly growing literature on the subject, see e.g.~\cite{AriKoc15,
BreHorMcKPlu06,CooKocPal07,Mir17,She18,BerBreLesMur18,WilZgl03}.

Given a numerical solution $\bar X$, the main idea of this approach is to construct an operator $T$ such that:
\begin{itemize}
\item fixed points of $T$ correspond to \emph{genuine} solutions;
\item $T$ is locally contracting, in a neighborhood of $\bar X$.
\end{itemize}
If such an operator can be defined, one can then hope to apply a fixed point argument, to show that $T$ has a fixed point in a neighborhood of $\bar X$. If this procedure is successful, we say that we have \emph{validated} the numerical solution $\bar X$, since we have proven the existence of a genuine solution close to $\bar X$. The following theorem gives some sufficient conditions to validate a numerical solution $\bar X$. Similar statements can be found in~\cite{AriKoc10,DayLesMis07,Plu01,Yam98} and in many subsequent works.

\begin{theorem}
\label{th:T}
Let $\left(\XX, \left\Vert\cdot  \right\Vert\right)$ be a Banach space and let $T:\XX\to\XX$ be a map of class $\CC^1$. Consider $\barX\in\XX$ and assume there exists a non negative constant $Y$ and a non negative, increasing and continuous function $Z$ such that
\begin{subequations}
\label{def:bounds_T}
\begin{align}
\label{def:Y_T}
\left\Vert T(\barX)-\barX\right\Vert &\leq Y \\
\label{def:Z_T}
\left\Vert \txtD T(X)\right\Vert &\leq Z\left(\Vert X-\barX \Vert\right),\quad\forall X\in\XX.
\end{align}
\end{subequations}
If there exists $\br>0$ such that
\begin{subequations}
\label{cond:r_T}
\begin{align}
\label{cond:rY_T}
Y+\int_0^{\br} Z(s)~\txtd s &<\br \\
\label{cond:rZ_T}
Z(\br) &<1
\end{align}
\end{subequations}
then $T$ has a unique fixed point in $\ball(\barX,\br)$, the closed ball of radius $\br$ centered at $\barX$.
\end{theorem}
\begin{proof}
We show that the conditions in~\eqref{cond:r_T} imply that $T$ is a contraction on $\ball(\barX,\br)$, which allows us to conclude by Banach's fixed point theorem. For $X\in \ball(\barX,\br)$, we estimate
\begin{align*}
\Vert T(X)-\barX \Vert &\leq \Vert T(X)-T(\barX) \Vert + \Vert T(\barX)-\barX \Vert \\
&\leq \int_0^1 \Vert \txtD T\left(\barX +t(X-\barX)\right)\Vert \Vert X-\barX\Vert~\txtd t + Y \\
&\leq \int_0^1 Z\left(t\Vert X-\barX\Vert\right) \Vert X-\barX\Vert ~\txtd t + Y \\
&\leq \int_0^1 Z\left(t\br\right) \br~\txtd t + Y \\
&= \int_0^{\br} Z\left(s\right)~\txtd s + Y,
\end{align*}
where we used to fact that $Z$ is increasing to get the last inequality. Therefore \eqref{cond:rY_T} yields that $T$ maps $\ball(\barX,\br)$ into itself, and using~\eqref{cond:rZ_T} we have that $T$ is indeed contracting on $\ball(\barX,\br)$.
\end{proof}

Before going further, let us make some comments about Theorem~\ref{th:T}. While it is a purely theoretical statement, we use it in a context where its application requires the help of the computer at several levels. Firstly, a computer is often mandatory to obtain an approximate solution. Secondly, the definition of $T$ will depend on $\bar X$, as one can only hope to get a contraction close to the solution, and thus the bounds $Y$ and $Z$ that we are going to derive will depend on $\bar X$ (this will become evident in Sections~\ref{sec:manifold} and~\ref{sec:connection}). To be precise, obtaining formulas for $Y$ and $Z$ satisfying~\eqref{def:bounds_T} is done by classical \emph{pen and paper} estimates, but these formulas will involve $\bar X$ in various ways, and thus their evaluation in order to check~\eqref{cond:r_T} requires a computer. To make a mathematically rigorous existence statement out of a numerical computation, one must necessarily control two types of errors:
\begin{itemize}
\item discretization or truncation errors, between the original problem and the finite dimensional approximation that the computer is working with;
\item round-off errors, coming from the fact that a computer can only use a finite subset of the real numbers.
\end{itemize}
The first issue is the most crucial one, and truncation errors must be estimated explicitly when deriving the bounds~\eqref{def:bounds_T} (again see Sections~\ref{sec:manifold} and~\ref{sec:connection} for explicit examples of such bounds). The issue of round-off errors can be handled quite easily thanks to the existence of many software packages and libraries which support interval arithmetic. In this work, we make use of \textsc{Matlab} with the \textsc{Intlab} package~\cite{Rum99}. We point out that one of the advantages of functional analytic methods is that they usually allow for the numerical solution to be computed with standard floating-point arithmetic, and only require interval arithmetic to evaluate the validation bounds, i.e.~to numerically but rigorously check~\eqref{cond:r_T}. Finally, let us point out that Theorem~\ref{th:T} provides a \emph{quantitative} existence statement, in the sense that the solution is proven to exists within an explicit distance, given by the validation radius $\br$ (which in practice is very small), of the known numerical solution $\barX$.

\begin{remark}
In Theorem~\ref{th:T}, assumption~\eqref{cond:rY_T} is in fact enough to ensure the existence of a locally unique fixed point of $T$ in a ball $\ball(\barX,\tilde r)$, for some $\tilde r \leq \br$. Indeed, first notice that~\eqref{cond:rY_T} implies that $Z(0)<1$. We then distinguish two cases. Either $Z(\br)<1$ and we are done. Or $Z(\br)\geq 1$, in which case there exists a unique $r_0\in(0,\br]$ such that $Z(r_0)=1$. We then consider the continuous function $P$ defined (for $r\geq 0$), by
\begin{equation*}
P(r)=Y+\int_0^r Z(s)~\txtd s -r,
\end{equation*}
which is negative at $r=\br$ by assumption~\eqref{cond:rY_T}. Since $P$ reaches its minimum at $r=r_0$, we also have $P(r_0)\leq P(\br)<0$. By continuity, there exists $\tilde r<r_0$ such that $P(\tilde r)<0$, and by definition of $r_0$ we also have $Z(\tilde r)<1$. Therefore, the assumptions~\eqref{cond:r_T} do indeed hold with $\tilde r$ instead of $\br$, which yields the existence of a unique fixed point of $T$ in $\ball(\barX,\tilde r)$.
\end{remark}

\section{The main steps}
\label{sec:main_steps}

We recall the main steps we use to validate a MEP:
\begin{enumerate}[label=(\Roman*)]
\item Locate and validate saddles and minima,
\item Introduce an equivalent polynomial vector field,
\item Validate heteroclinic orbits for this polynomial vector field:
\begin{enumerate}[label=(\alph*)]
\item Compute and validate the local unstable manifold of the saddles,
\item Validate trapping regions around the minima,
\item Compute and validate orbits connecting unstable manifolds and trapping regions.
\end{enumerate}
\end{enumerate}

In this section, we outline what is required for each step. While the method introduced is rather general, we feel that the exposition is made clearer by focusing on an explicit example. Therefore, we consider for the remainder of the paper the M\"uller-Brown potential~\cite{MulBro79}, which is a well-known test case for numerical methods in theoretical chemistry. It is given by

\begin{equation*}
V(x,y)=\sum_{i=1}^4 \alpha^{(i)} \exp\left(a^{(i)}\left(x-x^{(i)}_0\right)^2+b^{(i)}\left(x-x^{(i)}_0\right)\left(y-y^{(i)}_0\right)+c^{(i)}\left(y-y^{(i)}_0\right)^2\right),
\end{equation*}
$V:\R^2\ra \R$ where a standard set of parameters is
\begin{align*}
\alpha=(-200, -100, -170, 15),\quad x_0=(1, 0, -0.5, -1),\quad y_0=(0, 0.5, 1.5, 1) \\
a=(-1, -1, -6.5, 0.7),\quad b=(0, 0, 11, 0.6),\quad c=(-10, -10, -6.5, 0.7).
\end{align*}

\begin{figure}[h!]
\includegraphics[width=0.5\linewidth]{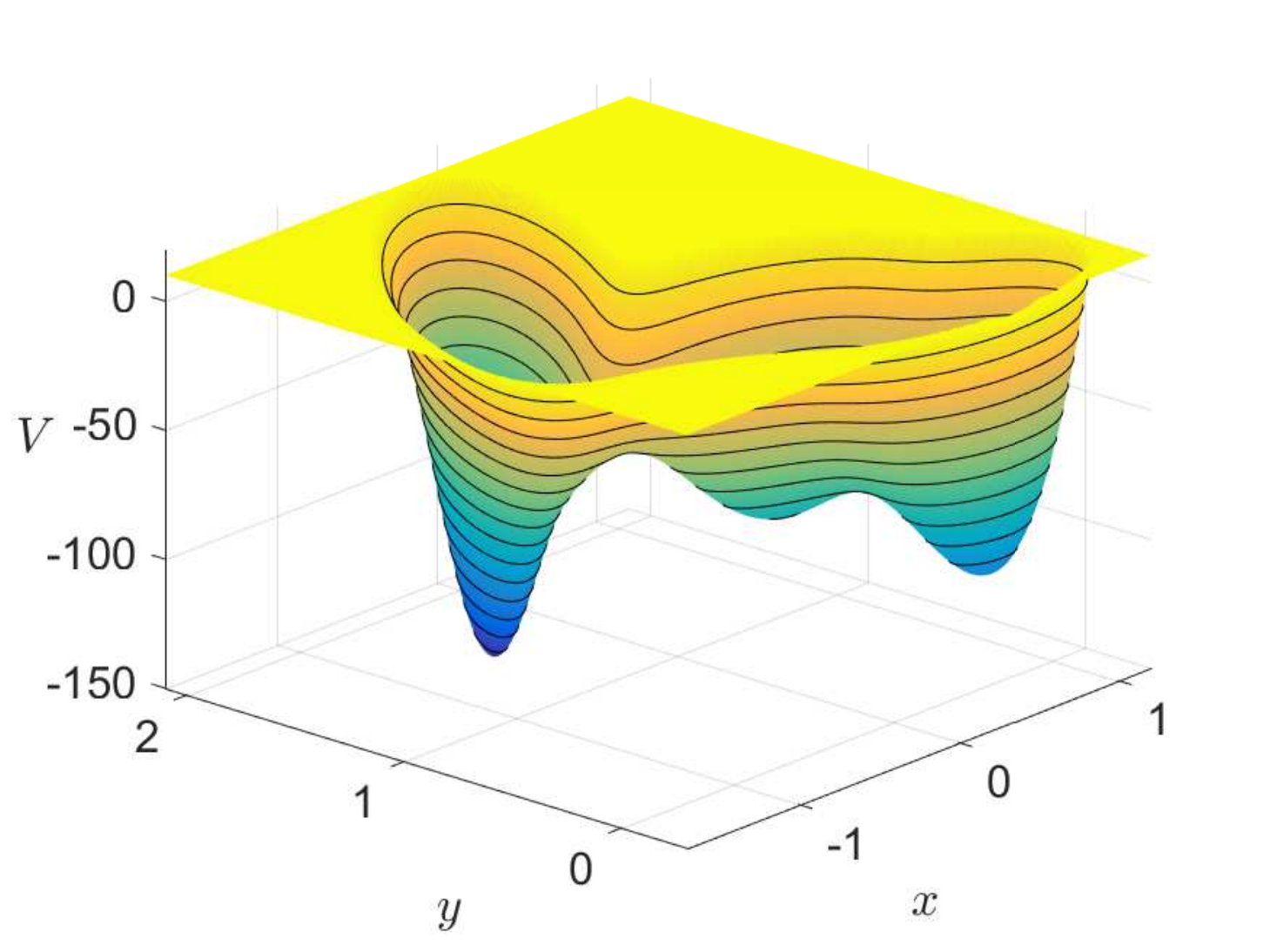}
\includegraphics[width=0.5\linewidth]{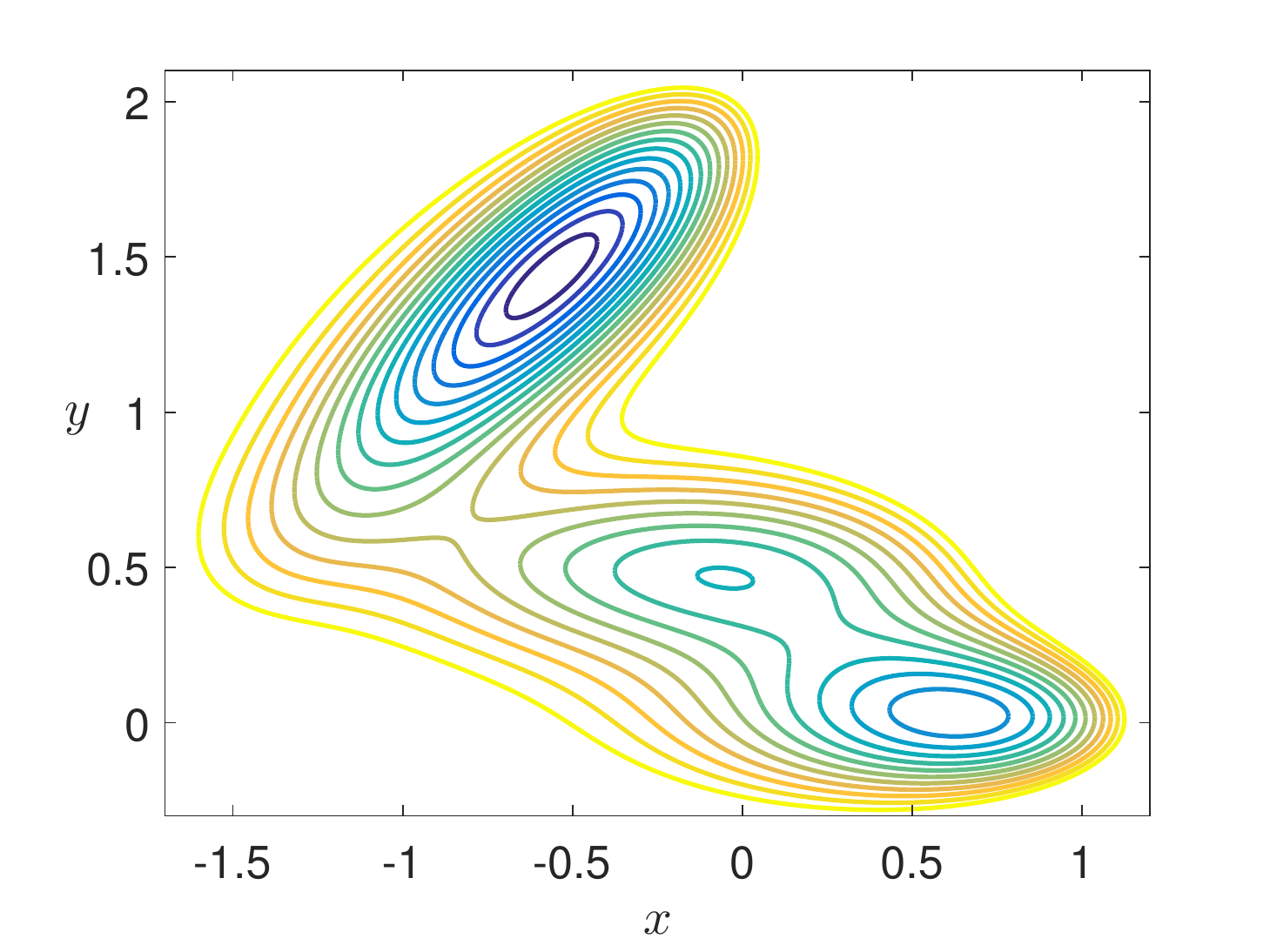}
\caption{The M\"uller-Brown potential (3D representation to the left, and level sets to the right).}
\end{figure}

\subsection{Finding and validating saddles points and local minima}
\label{sec:step_zeros}

For the M\"uller-Brown potential, numerically finding saddles and minima is rather straightforward, but for more complicated potentials this is where the numerical methods mentioned in Section~\ref{sec:background_stochastics} come into play. Once numerical approximations of saddles and minima have been found, one can readily validate such points, for instance by using the built-in \textsc{Intlab} function \texttt{verifynlss}, that provides a tight validated enclosure of solutions of a (finite dimensional) nonlinear system. Alternatively, one could use Theorem~\ref{th:T}. We are going to detail, how this can be done for one saddle point. This example is rather trivial because the problem is already finite dimensional: we want to validate a zero of $\nabla V:\R^2\to \R^2$, and hence there are no truncation errors to take care of. However, we believe that presenting this simple example in details will help the reader to better understand the more complicated applications of Theorem~\ref{th:T} that are to come in Sections~\ref{sec:manifold} and~\ref{sec:connection}. Using Newton's method, we get the following numerical approximation for the first saddle
\begin{equation*}
\barX=(-0.822001558732732, 0.624312802814871).
\end{equation*}
Our aim is now to prove, using Theorem~\ref{th:T}, that there is a zero of $\nabla V$ in a neighborhood of $\barX$. To do so, we introduce the operator
\begin{equation*}
T=I_2-A\nabla V,
\end{equation*}
where $I_2$ is the $2\times 2$ identity matrix, and
\begin{equation*}
A=\begin{pmatrix}
0.000085206327815 & 0.001664239983782 \\
0.001664239983782 & 0.000622806485951
\end{pmatrix}
\end{equation*}
is a numerically computed approximate inverse of the Hessian $\txtD^2 V(\barX)$. Therefore, $T$ is a Newton-like operator and should be a contraction around $\barX$. We now consider on $\R^2$ the supremum norm
\begin{equation*}
\left\Vert (x,y) \right\Vert =\max(\vert x\vert, \vert y\vert)
\end{equation*} 
and compute the bounds $Y$ and $Z$ satisfying~\eqref{def:bounds_T}. We point out that, while the choice of the norm is rather unimportant for this simple example, it will be crucial in the more involved cases presented later. To get a $Y$ bound satisfying~\eqref{def:Y_T}, we simply define
\begin{equation*}
Y=\left\Vert A\nabla V(\barX)\right\Vert.
\end{equation*}
Notice that the evaluation of this formula must be done with interval arithmetic, because we need to be sure that $\left\Vert T(\barX)-\barX\right\Vert \leq Y$. We get
\begin{equation*}
Y=5.291861481039345 \times 10^{-16}.
\end{equation*}
Next, we define a $Z$ bound satisfying~\eqref{def:Z_T}. In fact, the condition~\eqref{def:Z_T} only needs to hold in a neighborhood of $\barX$, therefore we introduce an a priori upper bound $r^*=10^{-5}$ on the validation radius, and aim at obtaining a function $Z$ such that
\begin{equation}
\label{def:Z_star}
\left\Vert \txtD T(X)\right\Vert \leq Z\left(\Vert X-\barX \Vert\right),\quad\forall X\in \ball(\barX,r^*).
\end{equation}
We will only have to check that the validation radius $\br$ that we eventually find satisfies $\br\leq r^*$. Based on the following splitting
\begin{align*}
\txtD T(X) &= I_2-A\txtD^2 V(X) \\
&= I_2-A\txtD^2 V(\barX) -A\left(\txtD^2 V(X) - \txtD^2 V(\barX)\right),
\end{align*}
we define (for $0\leq r\leq r^*$)
\begin{equation*}
Z(r)=Z_1+Z_2r,
\end{equation*}
where
\begin{equation*}
Z_1= \left\Vert I_2-A\txtD^2 V(\barX)\right\Vert \quad \text{and}\quad Z_2= \left\Vert A\right\Vert \sup_{X\in \ball(\barX,r^*)} \left\Vert \txtD^3 V(X)\right\Vert.
\end{equation*}
By the triangle inequality and the mean value theorem, we indeed have that~\eqref{def:Z_star} holds. Again, the evaluation of $Z_1$ and $Z_2$ must be done using interval arithmetic. Notice that (an upper bound of) $\sup_{X\in \ball(\barX,r^*)} \left\Vert \txtD^3 V(X)\right\Vert$ can be easily obtained with \textsc{Intlab}, by evaluating $\Vert \txtD^3 V(\tilde X)\Vert$, where $\tilde X = \ball(\barX,r^*)$.  We obtain
\begin{equation*}
Z_1=6.606610122939668\times 10^{-15} \quad \text{and}\quad Z_2=23.179491548050574.
\end{equation*}
To conclude, we must find $r>0$ satisfying~\eqref{cond:r_T}. The first condition may be re-written as
\begin{equation*}
\frac{Z_2}{2}r^2 - (1-Z_1)r + Y <0,
\end{equation*}
and the second simply as
\begin{equation*}
r <\frac{1-Z_1}{Z_2}.
\end{equation*}
We check, using interval arithmetic, that
\begin{equation*}
 \br=8.277841556061024\times 10^{-16}
\end{equation*} 
satisfies these two inequalities, and thus Theorem~\ref{th:T} yields the existence of a unique zero $X$ of $\nabla V$ such that $\left\Vert X-\barX \right\Vert\leq \br$. We have validated the numerical zero $\barX$.

\begin{remark}
Once a zero $\barX$ of $\nabla V$ has been validated, we can then apply the same techniques to rigorously compute its eigenvalues/eigenvectors, and thus obtain its Morse index. We omit the details here, as this problem is essentially identical to the one we just presented. We simply mention that an eigenvector problem naturally does not have a locally unique solution, and that on has to include a normalization condition to recover uniqueness and being able to use Banach fixed point theorem (see for instance~\cite{CasLes13}).
\end{remark}

\subsection{An equivalent formulation with a polynomial vector field}
\label{sec:polynomial_reformulation}

We are now almost ready to validate a heteroclinic orbit between a saddle and a minimum, for the vector field $-\nabla V$. Before doing so, we extend the dimension of the phase space, to obtain an \emph{equivalent} vector field with only polynomial nonlinearities. As we will see in Sections~\ref{sec:manifold} and~\ref{sec:connection}, this reformulation is a sort of trade-off. On the one hand, having polynomial nonlinearities makes it easier to obtain (and to implement) some of the estimates, whereas on the other hand, having more dimensions increases the computational cost. We point out that it is only the dimension of the phase space that is increased, and not the dimension of the object we have to validate, hence the increase in computational cost is not dramatic for low-dimensional examples. However, it is clear that this reformulation is not suitable for high-dimensional systems. We mention alternative approaches in Section~\ref{sec:generalizations}.

To obtain this new vector field with only polynomial nonlinearities, we introduce variables for the non polynomial terms in $-\nabla V$; that is we define
\begin{equation*}
\psi^{(i)}(x,y)=\alpha^{(i)} \exp\left(a^{(i)}\left(x-x^{(i)}_0\right)^2+b^{(i)}\left(x-x^{(i)}_0\right)\left(y-y^{(i)}_0\right)+c^{(i)}\left(y-y^{(i)}_0\right)^2\right), \quad i=1,\ldots,4,
\end{equation*}
consider $z^{(i)}=\psi^{(i)}(x,y)$ for $i=1,\ldots,4$, and compute the differential equations satisfied by these new quantities, assuming
\begin{equation}
\label{eq:grad_form}
\begin{pmatrix}
x' \\
y'
\end{pmatrix}
=-\nabla V(x,y).
\end{equation}
This leads us to consider the following \emph{extended variables}
\begin{equation}
\label{def:X}
X=(x,y,z^{(1)},z^{(2)},z^{(3)},z^{(4)})^\top,
\end{equation}
where the $z^{(i)}$ are now independent variables, and the associated \emph{extended system} is
\begin{equation}
\label{eq:extended_system}
X'=f(X),
\end{equation}
where the vector field is given by
\begin{align*}
f^{(1)}(X) &= -\sum_{i=1}^4\left(2a^{(i)}X^{(1)}+b^{(i)}X^{(2)}-w^{(i)}_1\right)X^{(i+2)} \\
f^{(2)}(X) &= -\sum_{i=1}^4\left(b^{(i)}X^{(1)}+2c^{(i)}X^{(2)}-w^{(i)}_2\right)X^{(i+2)} \\
f^{(j+2)}(X) &= -\left(\left(2a^{(j)}X^{(1)}+b^{(j)}X^{(2)}-w^{(j)}_1\right)\sum_{i=1}^4\left(2a^{(i)}X^{(1)}+b^{(i)}X^{(2)}-w^{(i)}_1\right)X^{(i+2)} \right. \\
&\qquad\quad \left. + \left(b^{(j)}X^{(1)}+2c^{(j)}X^{(2)}-w^{(j)}_2\right)\sum_{i=1}^4\left(b^{(i)}X^{(1)}+2c^{(i)}X^{(2)}-w^{(i)}_2\right)X^{(i+2)}\right)X^{(j+2)},
\end{align*}
for $j=1,\ldots,4$, and we define
\begin{equation*}
w^{(i)}_1=2a^{(i)}x^{(i)}_0+b^{(i)}y^{(i)}_0,\quad w^{(i)}_2=b^{(i)}x^{(i)}_0+2c^{(i)}y^{(i)}_0.
\end{equation*}
The formulations~\eqref{eq:grad_form} and~\eqref{eq:extended_system} are \emph{equivalent} in the following sense. The system~\eqref{eq:extended_system} was designed so that, if $(x,y)^\top$ is a solution of~\eqref{eq:grad_form} then $X$ defined as in~\eqref{def:X} with $z=\psi(x,y)$ solves~\eqref{eq:extended_system}. Conversely, by the uniqueness statement of the Picard-Lindel\"of Theorem, we get that a solution of~\eqref{eq:extended_system} with a suitable initial (or asymptotic) condition gives a solution of~\eqref{eq:grad_form}. 
\begin{lemma}
\label{lem:equivalence}
Let $X=(x,y,z^{(1)},z^{(2)},z^{(3)},z^{(4)})^\top$ be a solution of~\eqref{eq:extended_system}. Assume one of the following conditions holds:
\begin{enumerate}[label=(\roman*)]
\item there exists a time $t_0$ for which $z(t_0)=\psi\left(x(t_0),y(t_0)\right)$;
\item $X_\infty=lim_{t\to -\infty} X(t)$ exists and is such that $z_\infty=\psi\left(x_\infty,y_\infty\right)$.
\end{enumerate}
Then, $(x,y)$ solves~\eqref{eq:grad_form}.
\end{lemma}
\begin{proof}
We can write $-\nabla V(x,y)=g(x,y,\psi(x,y))$, where $g:\R^6\to\R^2$,
\begin{align*}
g^{(1)}(x,y,z) &= -\sum_{i=1}^4\left(2a^{(i)}x+b^{(i)}y-w^{(i)}_1\right)z^{(i)} \\
g^{(2)}(x,y,z) &= -\sum_{i=1}^4\left(b^{(i)}x+2c^{(i)}y-w^{(i)}_2\right)z^{(i)}.
\end{align*}
We can then write
\begin{equation*}
f(X)=\begin{pmatrix}
g(x,y,z) \\
\txtD \psi(x,y)g(x,y,z)
\end{pmatrix},
\end{equation*}
and given a solution $X=(x,y,z^{(1)},z^{(2)},z^{(3)},z^{(4)})^\top$ of~\eqref{eq:extended_system}, $(x,y)$ solves~\eqref{eq:grad_form} if and only if $z(t)=\psi(x(t),y(t))$ for all time $t$ for which the solution is defined. However, since $X$ solves~\eqref{eq:extended_system} we have
\begin{equation*}
z'=\txtD \psi(x,y)g(x,y,z)=\txtD \psi(x,y)\begin{pmatrix} x' \\ y' \end{pmatrix} = \frac{\txtd}{\txtd t}\psi(x,y)
\end{equation*}
and thus either $(i)$ or $(ii)$ ensures that $z(t)=\psi(x(t),y(t))$ for all time $t$ for which the solution is defined.
\end{proof}

We also have the following statement which shows that for any equilibrium point of~\eqref{eq:grad_form} going to the extended system~\eqref{eq:extended_system} only adds center directions.
\begin{lemma}
\label{lem:extended_spectrum}
Let $(x_0,y_0)$ be a zero of $\nabla V$, and $X_0$ be defined as in~\eqref{def:X} with $z_0=\psi(x_0,y_0)$. Then 
\begin{equation*}
\det(\lambda I_6-\txtD f(X_0))=\lambda^4\det(\lambda I_2+\txtD^2V(x_0,y_0)).
\end{equation*}
In particular, $\txtD f(X_0)$ has the same spectrum as $-\txtD^2V(x_0,y_0)$, up to an additional $0$ of algebraic multiplicity 4.
\end{lemma}
\begin{proof}
With the same notations as in the previous proof, we have
\begin{equation*}
-\nabla V(x,y)=g(x,y,\psi(x,y)) \quad \text{and} \quad
f(X)= \begin{pmatrix}
I_2 \\
\txtD \psi(x,y)
\end{pmatrix}g(x,y,z).
\end{equation*}
If $(x_0,y_0)$ is such that $\nabla V(x_0,y_0)=0$, and $z_0=\psi(x_0,y_0)$, then $g(X_0)=0$ and hence
\begin{equation*}
\txtD f(X_0)=\begin{pmatrix}
I_2 \\
\txtD \psi(x_0,y_0)
\end{pmatrix}\txtD g(X_0),
\end{equation*}
whereas
\begin{equation*}
-\txtD^2V(x_0,y_0)= \txtD g(X_0)\begin{pmatrix}
I_2 \\
\txtD \psi(x_0,y_0)
\end{pmatrix}.
\end{equation*}
The result follows from the identity
\begin{equation*}
\det(MN)=\det(NM),
\end{equation*}
with
\begin{equation*}
M=\begin{pmatrix}
\lambda I_2 & \txtD g(X_0)\\
\begin{pmatrix}
I_2 \\
\txtD \psi(x_0,y_0)
\end{pmatrix} & I_6
\end{pmatrix} \qquad 
N=\begin{pmatrix}
I_2 & -\txtD g(X_0)\\
0 & \lambda I_6
\end{pmatrix}.
\end{equation*}
\end{proof}

This \emph{polynomial reformulation} is based on ideas from automatic differentiation, and was introduced in the context of a posteriori validation in~\cite{LesMirRan16}. See~\cite[Section~4.2]{KepMir18} for a discusion about this type of polynomial reformulation in a more general framework, and also~\cite{Knu97,HarCanFigLuqMon16} for detailed algorithmic descriptions.

\subsection{Validation of heteroclinic orbits}

We are now ready to validate heteroclinic orbits. Since these orbits features boundary conditions \emph{at infinity}, we use the method of \emph{projected boundaries}~\cite{AscMatRus94}. That is, given a saddle point $S$ and a minimum $M$, we are going to look for an orbit $X:[0,\tau]\to\R^6$ satisfying
\begin{equation*}
X'=f(X),\quad X(0)\in W^{\textnormal{u}}_{\textnormal{loc}}(S),\quad X(\tau)\in W^{\textnormal{s}}_{\textnormal{loc}}(M).
\end{equation*}
Therefore, we first have to compute and validate a local unstable manifold for $S$ and a local stable manifold for $M$, and then compute and validate an orbit that connects them. The fact that $M$ is a minimum makes the determination of a local stable manifold rather straightforward, as will be explained just below. The validation a local unstable manifold for the saddle, and of an orbit connecting this unstable manifold to the minimum are more involved. We briefly present the main ideas in subsequent paragraphs, and postpone the technicalities to Sections~\ref{sec:manifold} and~\ref{sec:connection} where all the needed estimates are derived.

\subsubsection{Validation of trapping regions around the minima}
\label{sec:trapping}

In our example, each considered minimum is strict, and therefore its stable manifold contains a whole neighborhood of the minimum. We only have to explicitly find such a neighborhood. Using \textsc{Intlab}, we can prove the existence of a small square around each minimum, such that
\begin{enumerate}
\item $V$ is strictly convex in the square,
\item $-\nabla V$ is pointing strictly inward, on the boundary of the square.
\end{enumerate}
The second property ensures that an orbit entering the square never leaves it again, and the first proves that there is no other minimum in the square, and hence that any orbit entering the square converges to the minimum. Such a square is thus part of the stable manifold of the minimum.

\begin{remark}
A more precise description of the dynamics near the minimum could be obtained by applying the techniques described in Section~\ref{sec:manifold} to compute a validated parameterization of the local stable manifold. This approach could of course also be used to validate a local stable manifold of a saddle point, which would then serves as an \emph{end point} for an heteroclinic orbit between two saddles points.
\end{remark}

\subsubsection{Computation and validation of a local unstable manifold for the saddles}
\label{sec:step_manifold}

We now explain the main ideas that we use to compute and validate a local unstable manifold for each saddle. Our technique is based on the \emph{parameterization method}, introduced in~\cite{CabFonLla03,CabFonLla03bis,CabFonLla05} (see also the recent book~\cite{HarCanFigLuqMon16}). We want to obtain a parameterization that conjugates the dynamics on the unstable manifold with the unstable dynamics of the linearized system. To be more precise, let us consider a saddle point $S$, in the extended phase space $\R^6$, and denote by $\lambda$ the unstable eigenvalue, together with an associated eigenvector $v$. As in Section~\ref{sec:step_zeros}, $\lambda$ and $v$ can first be computed numerically and then validated rigorously, using Theorem~\ref{th:T} or via the built-in \texttt{verifyeig} function of \textsc{Intlab}. We look for a parameterization $p:[-1,1]\to \R^6$ such that $p(0)=S$ and
\begin{equation}
\label{eq:conjugacy}
\phi_t(p(\theta))=p(e^{\lambda t}\theta),\quad \forall~\theta\in[-1,1],\ \forall~t\leq 0, 
\end{equation}
where $\phi_t$ is the flow generated by $f$ (see Figure~\ref{fig:PM}). 
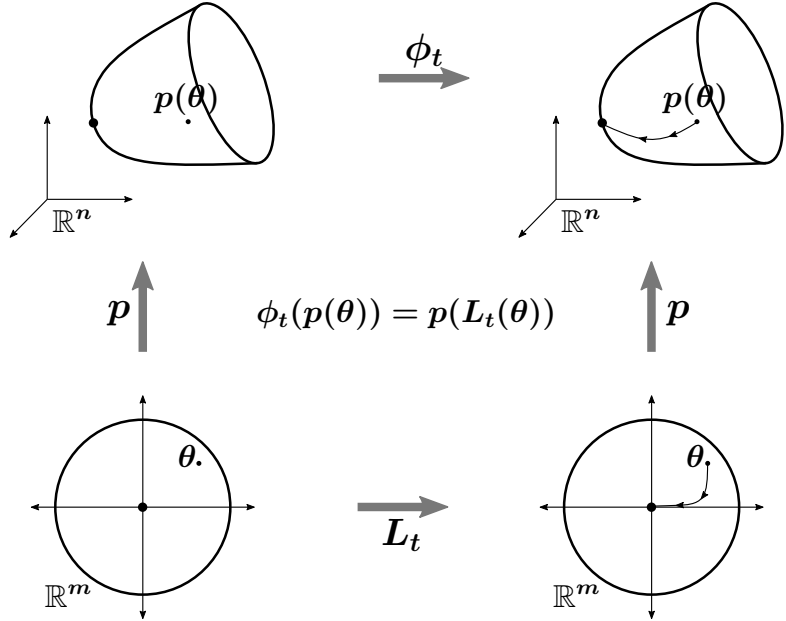
\begin{figure}[h]
\centering
\begin{overpic}
[width=0.65\linewidth]{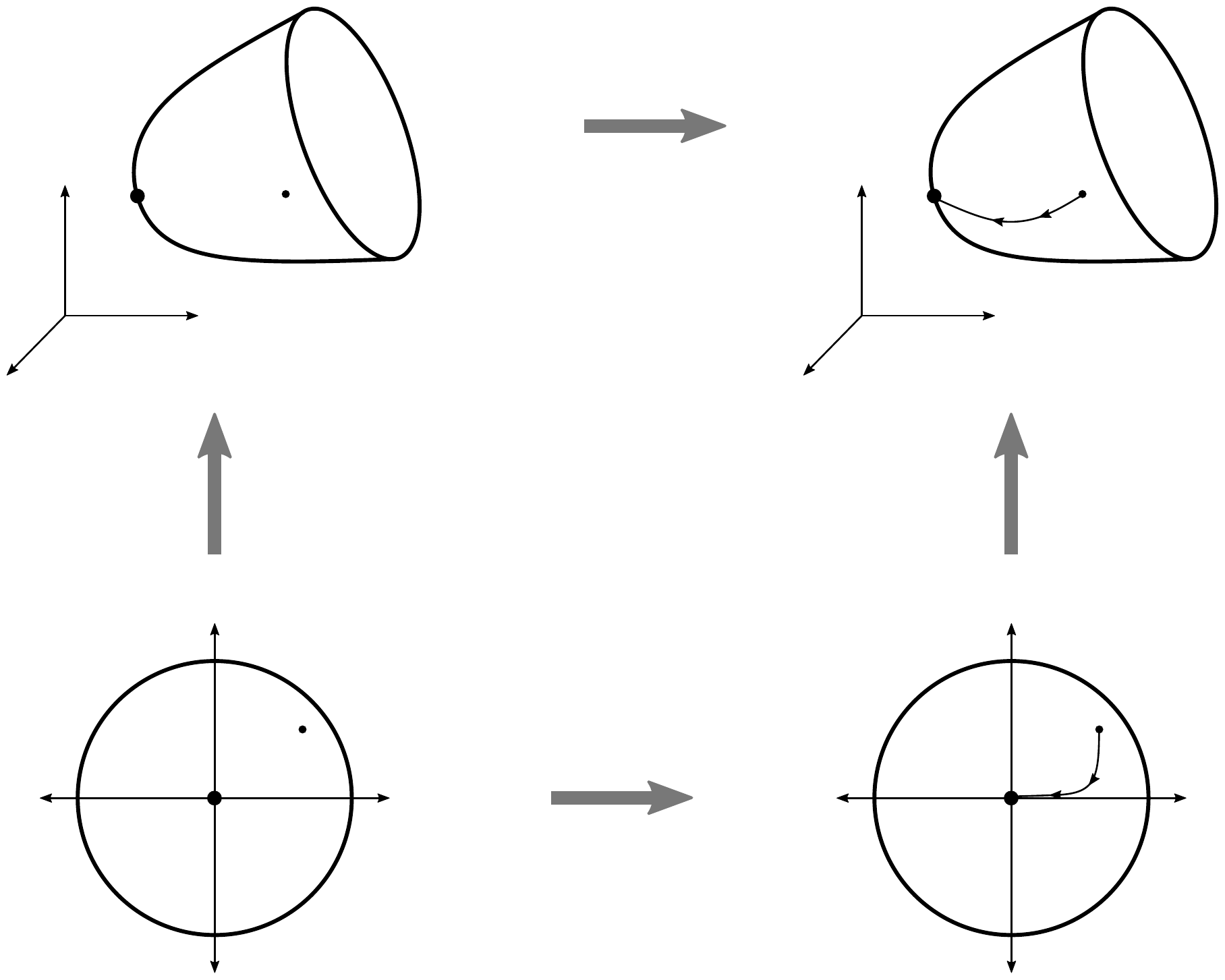}
\put (5,2) {\large$\bm{\R^m}$}
\put (70,2) {\large$\bm{\R^m}$}
\put (6,50) {\large$\bm{\R^n}$}
\put (71,50) {\large$\bm{\R^n}$}
\put (22,20) {\large$\bm{\theta}$}
\put (87,20) {\large$\bm{\theta}$}
\put (19,66) {\large$\bm{p(\theta)}$}
\put (84,66) {\large$\bm{p(\theta)}$}
\put (13,39) {\Large$\bm{p}$}
\put (84.5,39) {\Large$\bm{p}$}
\put (51,72) {\Large$\bm{\phi_t}$}
\put (48,10) {\Large$\bm{L_t}$}
\put (32,38.5) {\large$\bm{\phi_t(p(\theta))=p(L_t(\theta))}$}
\end{overpic}
\caption{Schematic illustration of the parameterization method. We want the parameterization $p$ to conjugate the nonlinear flow $\phi$ to the linearized flow $L$. In our actual example the (extended) phase space is of dimension $n=6$, whereas the parameter space is of dimension $m=1$ (the unstable manifold is one-dimensional), and the linearized flow is given by $L_t\theta = e^{\lambda t}\theta$. As suggested by the illustration, this method is not restricted to one-dimensional manifolds.}
\label{fig:PM}
\end{figure}
However, this formulation is not the most convenient one to work with, because it involves the flow. To get rid of it, one can take a time derivative of the above equation and evaluate at $t=0$, to obtain the following \emph{invariance equation}
\begin{equation}
\label{eq:invariance_equation}
f(p(\theta))=p'(\theta)\lambda\theta ,\quad \forall~\theta\in[-1,1].
\end{equation}
One can check that, if $p$ is such that $p(0)=S$ and solves~\eqref{eq:invariance_equation}, then $p$ satisfies~\eqref{eq:conjugacy}, therefore $p([-1,1])$ is a local unstable manifold of $S$. The invariance equation~\eqref{eq:invariance_equation} is the one we are going to solve, first numerically and then rigorously by applying a variation of Theorem~\ref{th:T}. Let us already mention that, since Theorem~\ref{th:T} is based on a contraction argument, it can only be used to validate locally unique solutions. However, the invariance equation~\eqref{eq:invariance_equation} together with the condition $p(0)=S$ has a one parameter family of solutions (corresponding to different rescalings) and to isolate the solution one has to add a constraint on the norm of the derivative, i.e. fix a scalar $\gamma$ such that $p'(0)=\gamma v$.

Since $f$ is analytic, there is an analytic solution to~\eqref{eq:invariance_equation} (see~\cite{CabFonLla03}), therefore we can look for a power series representation of $p$. Starting from $p(0)=S$ and $p'(0)=\gamma v$, and taking advantage of the fact that our extended vector field $f$ is polynomial, it becomes rather straightforward to recursively compute the coefficients of the Taylor expansion. Once sufficiently many coefficients have been computed, the remaining \emph{tail} can be controlled by Theorem~\ref{th:T}. A detailed exposition of this procedure is the subject of Section~\ref{sec:manifold}. Before proceeding further, let us mention that the approach presented here, i.e., combining the parameterization method with a posteriori validation techniques, was first used in~\cite{BerLesMirMis11} and then further developed in several subsequent works (among which~\cite{BreLesMir16} where the role of the scaling of the eigenvectors is investigated in details, and~\cite{BerMirRei16}, where an extension to treat \emph{resonance conditions} is introduced).

\subsubsection{Computation and validation of orbits connecting unstable manifolds and trapping regions}

Our last objective is to validate an orbit starting on the unstable manifold of a saddle and ending in the attracting set of a minimum, that is we want validate a solution $X:[0,\tau]\to \R^6$ of the problem
\begin{equation}
\label{eq:cauchy_problem}
\left\{
\begin{aligned}
X'&=f(X) \\
X(0)&=p(1)
\end{aligned}
\right.
\end{equation}
for a time $\tau$ large enough so that $\left(X^{(1)}(\tau),X^{(2)}(\tau)\right)$ lies in the validated trapping region of a minimum of the potential $V$. Notice that, $X(0)=p(1)$ ensures that the orbit is on the unstable manifold of the saddle (in the extended phase space), and by Lemma~\ref{lem:equivalence} $(ii)$ we get that $\left(X^{(1)},X^{(2)}\right)$ indeed solves~\eqref{eq:grad_form} and is on the unstable manifold of the saddle in the original two-dimensional phase space. To validate a solution $X$ of~\eqref{eq:cauchy_problem}, we use an approach based on Chebyshev series, introduced in~\cite{LesRei14}. The method is similar in spirit to the one described in the previous subsection to solve the invariance equation, the main difference being that we look for a solution represented as a Chebyshev series and not as a power series. Again, our approximate solution will only have a finite number of coefficients, and we control the error with the true solution thanks to a variation of Theorem~\ref{th:T}. One difference with the situation described in the previous subsection is that the coefficients cannot be computed recursively, as the obtained system is fully coupled. Therefore, we have to consider a fixed point operator on the whole sequence of coefficients rather than just of the tail, but this operator has no reason to be a contraction. We remedy to this by introducing a Newton-like reformulation. 

In practice, the integration time $\tau$ for which we have to validate the orbit until it reaches the trapping region of the minimum can be quite long, and maybe too long for the orbit to be validated using a single Chebyshev series. To validate the orbit for longer times, we use the \emph{domain decomposition} principle introduced in the context of a posteriori validation using Chebyshev series in~\cite{BerShe15}. The procedure outlined in this subsection is presented in full details in Section~\ref{sec:connection}.

\section{Validation of local manifolds}
\label{sec:manifold}

\subsection{Background}
\label{sec:background_manifold}

The material introduced in this subsection is mostly standard in rigorous numerics. It is mainly included to fix the notation and to select/refine suitable operator norms for our computations. The reader familiar with the notions presented here might skip Section~\ref{sec:background_manifold} on first reading, proceed to Section~\ref{ssec:framework4}, and only refer back when needed.

\subsubsection{Notations}

For $X\in\R^6$, we denote its components by $X^{(i)}$, $i=1,\ldots,6$. For $\bu\in\R^\N$, we denote its components by $\bu_n$, $n\in\N$. Given two sequences $\bu,\bv\in\R^\N$, we denote by $\bu\star \bv$ their Cauchy product
\begin{equation*}
(\bu\star \bv)_n=\sum_{m=0}^n \bu_m \bv_{n-m},\quad \forall~n\in\N. 
\end{equation*} 
For $\p\in \left(\R^\N\right)^6$, using the isomorphism between $\left(\R^6\right)^\N$ and $\left(\R^\N\right)^6$, we can think of $\p$ either as a sequence indexed on $\N$ and with values in $\R^6$, that is
\begin{equation*}
\p=\left(\p_n\right)_{n\in\N},\quad \p_n\in\R^6,
\end{equation*}
or as $6$ sequences indexed on $\N$ and with values in $\R$, that is
\begin{equation*}
\p=\left(\p^{(1)},\ldots,\p^{(6)}\right),\quad \p^{(i)}\in\R^\N.
\end{equation*}
To a sequence $\p\in\left(\R^6\right)^\N$ can be associated a power series $p$ with values in $\R^6$, defined by
\begin{equation*}
p(\theta)=\sum_{n=0}^\infty \p_n\theta^n,\quad \theta\in\R.
\end{equation*}
We always use this convention of denoting a (possibly multivariate) function like $p$ with an unbold character, whereas we use a bold character like $\p$ for the associated sequence of coefficients. We introduce in the next subsection a subspace of $\left(\R^6\right)^\N$ for which such series is guaranteed to converge, at least for $\theta\in[-1,1 ]$. Similarly, we denote by $\f:\left(\R^\N\right)^6 \to \left(\R^\N\right)^6$ the map such that, for any power series $p$, $\f(\p)$ are the coefficients of the power series $f(p)$, that is 
\begin{equation*}
\f^{(1)}_n(\p)=  -\sum_{i=1}^4\left(2a^{(i)}\left(\p^{(1)}\star \p^{(i+2)}\right)_n+b^{(i)}\left(\p^{(2)}\star \p^{(i+2)}\right)_n-w^{(i)}_1\p^{(i+2)}_n\right),
\end{equation*}
and so on.

\subsubsection{Norms and Banach spaces}

For $\eta\in \R_{>0}^6$, we define the following weighted $1$-norm on $\R^6$ :
\begin{equation*}
\left\vert X\right\vert_\eta = \sum_{i=1}^6 \vert X^{(i)} \vert \eta^{(i)}.
\end{equation*}
We denote the usual $\ell^1$-norm on $\R^\N$ by $\left\Vert\cdot\right\Vert_1$, that is
\begin{equation*}
\left\Vert \bu \right\Vert_1 = \sum_{n=0}^\infty \vert \bu_n \vert.
\end{equation*}
We recall that
\begin{equation*}
\left\Vert \bu\star\bv \right\Vert_1 \leq \left\Vert \bu \right\Vert_1 \left\Vert \bv \right\Vert_1.
\end{equation*}
To highlight the key space we are going to use, we state it in the next definition:

\begin{definition}
\label{def:space_manifold}
Let $\eta\in \R_{>0}^6$. For $\p\in \left(\R^\N\right)^6$ we define
\begin{equation*}
\left\Vert \p \right\Vert_{\X_\eta} = \sum_{i=1}^6 \sum_{n=0}^\infty  \vert\p_n^{(i)}\vert \eta^{(i)},
\end{equation*}
and
\begin{equation*}
\X_\eta=\left\{\p\in \left(\R^\N\right)^6 ,\ \left\Vert \p\right\Vert_{\X_\eta}<\infty\right\}.
\end{equation*}
\end{definition}

Notice that we have
\begin{equation*}
\left\Vert \p \right\Vert_{\X_\eta} = \sum_{i=1}^6 \left\Vert \p^{(i)} \right\Vert_1 \eta^{(i)},
\end{equation*}
and with a slight abuse of notation
\begin{equation*}
\left\Vert \p \right\Vert_{\X_\eta} = \left\vert \left\Vert \p \right\Vert_1\right\vert_\eta,
\end{equation*}
where $\left\Vert \p \right\Vert_1$ must be understood as $\left(\left\Vert \p^{(1)} \right\Vert_1,\ldots,\left\Vert \p^{(6)} \right\Vert_1\right)$.

\subsubsection{Operator norms}
\label{sec:op_blocks_manifold}

If $A=\left(A^{(i,j)}\right)_{1\leq i,j\leq 6}$ is a $6\times 6$ matrix, we still denote by $\left\vert A\right\vert_\eta$ the associated operator norm, that is
\begin{equation*}
\left\vert A\right\vert_\eta=\sup\limits_{\left\vert X\right\vert_\eta=1} \left\vert AX\right\vert_\eta.
\end{equation*}
We recall that
\begin{equation*}
\left\vert A\right\vert_\eta = \max\limits_{1\leq j\leq 6} \frac{1}{\eta^{(j)}}\sum_{i=1}^6 \left\vert A^{(i,j)} \right\vert \eta^{(i)}.
\end{equation*}
For a given matrix $A$ with positive coefficients, the weight $\eta$ can be chosen in a optimal way to minimize $\left\vert A \right\vert_\eta$. This is the content of the following proposition.

\begin{proposition}
\label{prop:frob}
Let $A$ be a $d \times d$ matrix of positive numbers. There exists a left-eigenvector $\bar\eta$, with positive coefficients, associated to the spectral radius of $A$, i.e.
\begin{equation}
\label{eq:frob}
\bar\eta^\top A = \rho(A)\bar\eta^\top.
\end{equation}
Besides
\begin{equation}
\label{eq:optimal_weight}
\left\vert A\right\vert_{\bar\eta} = \rho(A) = \inf\limits_{\eta\in \R_{>0}^d} \left\vert A\right\vert_{\eta}.
\end{equation}
\end{proposition}
\begin{proof}
The existence of $\bar\eta$ is given by the Perron-Frobenius Theorem. Identity~\eqref{eq:frob} then yields 
\begin{equation*}
\sum_{i=1}^6  A^{(i,j)} \eta^{(i)} = \rho(A) \eta^{(j)},\quad \forall~j\in\{1,\ldots,d\},
\end{equation*}
and thus $\left\vert A\right\vert_{\bar\eta} = \rho(A)$. Finally, since the spectral radius is bounded from above by any operator norm, $\bar\eta$ is indeed a weight that minimizes $\left\vert A\right\vert_{\eta}$. 
\end{proof}

\begin{remark}
The above proposition allows us to choose weights in an optimal way, at least with respect to one crucial estimate needed to control the contraction rate and apply Theorem~\ref{th:T} (see Sections~\ref{sec:Z1_manifold} and~\ref{sec:Z1_orbit}). We point out that the connection between the choice of such weigths and the Perron-Frobenius Theorem was also noticed in~\cite{BerWil17}. An alternative approach would be to include the weights in the validation radius $\bar r$ of Theorem~\ref{th:T} (see for instance~\cite{Ber18}).
\end{remark}

If $\A=\left(\A_{m,n}\right)_{m,n\in\N}$ is an "infinite matrix" representing a linear operator on $\ell^1$, we still denote by $\left\Vert \A\right\Vert_1$ the associated operator norm, that is
\begin{equation*}
\left\Vert \A\right\Vert_1=\sup\limits_{\left\Vert \bu\right\Vert_1=1} \left\Vert \A\bu\right\Vert_1.
\end{equation*}
We recall that
\begin{equation*}
\left\Vert \A\right\Vert_1 = \sup\limits_{n\in\N} \sum_{m\in\N} \left\vert \A_{m,n} \right\vert.
\end{equation*}
Finally, a linear operator $\A$ on $\X_\eta$ can be represented as an "infinite block-matrix", that is
\begin{equation*}
\A=\begin{pmatrix}
\A^{(1,1)} & \ldots & \A^{(1,6)} \\
\vdots & \ddots & \vdots \\ 
\A^{(6,1)} & \ldots & \A^{(6,6)} \\  
\end{pmatrix},
\end{equation*}
where each block $\A^{(i,j)}$ is a linear operator on $\ell^1$. Each of these blocks can themselves be represented as infinite matrices $\A^{(i,j)}=\left(\A^{(i,j)}_{m,n}\right)_{m,n\in\N}$ and their $\ell^1$ operator norm is then given by
\begin{equation*}
\left\Vert \A^{(i,j)}\right\Vert_1 = \sup_{n\in\N} \sum_{m\in\N} \left\vert \A^{(i,j)}_{m,n}\right\vert.
\end{equation*} 
By a slight abuse of notation, when $\A$ is a linear operator on $\X_\eta$, we consider that $\left\Vert \cdot \right\Vert_1$ applies block-wise to $\A$, that is we define\begin{equation*}
\left\Vert \A \right\Vert_1=\begin{pmatrix}
\left\Vert \A^{(1,1)}\right\Vert_1 & \ldots & \left\Vert \A^{(1,6)}\right\Vert_1 \\
\vdots & \ddots & \vdots \\ 
\left\Vert \A^{(6,1)}\right\Vert_1 & \ldots & \left\Vert \A^{(6,6)}\right\Vert_1 \\  
\end{pmatrix}.
\end{equation*}
Similarly, $\left\vert \cdot \right\vert_{\eta}$ applies component-wise to $\left\Vert \A \right\Vert_1$, that is we define
\begin{equation*}
\left\vert \left\Vert \A \right\Vert_1 \right\vert_\eta = \max_{1\leq j\leq 6} \frac{1}{\eta^{(j)}}\sum_{i=1}^6\left\Vert \A^{(i,j)}\right\Vert_1 \eta^{(i)}.
\end{equation*}
We have that
\begin{align*}
\left\Vert \A\right\Vert_{\X_\eta} &= \max\limits_{1\leq j\leq 6} \sup\limits_{n\in\N} \frac{1}{\eta^{(j)}} \sum_{i=1}^6\sum_{m\in\N} \left\vert \A^{(i,j)}_{m,n}\right\vert \eta^{(i)} \\
&\leq \max\limits_{1\leq j\leq 6}  \frac{1}{\eta^{(j)}} \sum_{i=1}^6 \eta^{(i)} \sup\limits_{n\in\N}\sum_{m\in\N} \left\vert \A^{(i,j)}_{m,n}\right\vert  \\
&= \max\limits_{1\leq j\leq 6}  \frac{1}{\eta^{(j)}} \sum_{i=1}^6  \left\Vert \A^{(i,j)}\right\Vert_1 \eta^{(i)},
\end{align*}
that is
\begin{equation}
\label{eq:lin_op_norm_computable}
\left\Vert \A\right\Vert_{\X_\eta} \leq \left\vert \left\Vert \A \right\Vert_1 \right\vert_\eta.
\end{equation}
Notice that, once $\left\Vert \A \right\Vert_1$ has been computed, Proposition~\ref{prop:frob} can be used to find a weight $\eta$ that minimize $\left\vert \left\Vert \A \right\Vert_1 \right\vert_\eta$. Similarly, if $\B$ is an $k$-linear operator on $\X_\eta$, we still denote by $\left\Vert \B\right\Vert_{\X_\eta}$ its operator norm, defined as
\begin{equation*}
\left\Vert \B\right\Vert_{\X_\eta} = \sup\limits_{\substack{\left\Vert \bu_l\right\Vert_{\X_\eta}=1 \\ l=1,\ldots,k }} \left\Vert \B(\bu_1,\ldots,\bu_k) \right\Vert_{\X_\eta},
\end{equation*}
and have that
\begin{equation}
\label{eq:multilin_op_norm_computable}
\left\Vert \B\right\Vert_{\X_\eta} \leq \max\limits_{1\leq j_1,\ldots,j_k\leq 6}  \frac{1}{\eta^{(j_1)}\ldots\eta^{(j_k)}} \sum_{i=1}^6  \left\Vert \B^{(i,j_1,\ldots,j_k)}\right\Vert_1 \eta^{(i)}.
\end{equation}

\subsection{Framework}
\label{ssec:framework4}

Let $(x_0,y_0)\in\R^2$ be a saddle point of $V$, denote by $\lambda>0$ the unstable eigenvalue and by $\mu<0$ the stable one. We want to validate a parametrization of the local unstable manifold. We consider $X_0$ defined as in~\eqref{def:X} with $z_0=\psi(x_0,y_0)$ and work with the extended system, given by $f$. By Lemma~\ref{lem:extended_spectrum}, $\lambda$ is also the unique unstable eigenvalue of $\txtD f(X_0)$. We denote by $v$ an associated eigenvector. We look for a parametrization $p:\R\to\R^6$ of the local unstable manifold at $X_0$, that satisfies $p(0)=X_0$, $p'(0)=\gamma v$, for some $\gamma>0$ to be chosen later, and the invariance equation
\begin{equation}
\label{eq:invariance}
p'(\theta)\lambda\theta=f(p(\theta)), \quad \theta\in[-1,1].
\end{equation}
We look for a power series representation of $p$, that is
\begin{equation}
\label{def:p}
p(\theta)=\sum_{n=0}^\infty \p_n \theta^n, \quad \theta\in[-1,1],\ \p_n\in\R^6.
\end{equation}
The invariance equation~\eqref{eq:invariance} can be rewritten as an equation on $\p$, yielding: 
\begin{equation}
\label{eq:invariance_coeff}
\p_0=X_0,\quad \p_1=\gamma v\quad \text{and}\quad n\lambda \p_n = \f_n(\p),\quad \forall n\geq 2.
\end{equation}
We denote $\pi_n(\p)=(\p_0,\ldots,\p_n)\in \left(\R^6\right)^{n+1}$, which we also may identify with the element 
\begin{equation*}
(\p_0,\ldots,\p_n,0,\ldots,0,\ldots)\in\left(\R^6\right)^{\N}. 
\end{equation*}
Notice that, for any $\bu,\bv\in \R^\N$, the $n$-th coefficient of a Cauchy product $(\bu\star \bv)_n$ only depends on the coefficients $\left(\bu_m\right)_{0\leq m\leq n}$ and $\left(\bv_m\right)_{0\leq m\leq n}$, and therefore $\f_n(\p)$ in fact only depends on $\pi_n(\p)$, i.e.
\begin{align*}
\f_n(\p) &=\f_n(\pi_n(\p)),\quad \forall n\geq 1.
\end{align*}
Next, we want to extract from $\f_n(\p)$ the contribution of $\p_n$. To do so, we consider the Taylor series representing $\f(\pi_n(\p))$ that is
\begin{equation*}
f\left(\sum_{k=0}^n p_k\theta ^k\right),
\end{equation*}
and use a Taylor expansion to write
\begin{align*}
f\left(\sum_{k=0}^n p_k\theta ^k\right) &= f\left(\sum_{k=0}^{n-1} p_k\theta ^k + p_n\theta^n\right) \\
&= f\left(\sum_{k=0}^{n-1} p_k\theta ^k \right) + Df\left(\sum_{k=0}^{n-1} p_k\theta ^k\right)p_n\theta^n + \text{higher order terms}.
\end{align*}
Looking at the coefficient of degree $n$ in each of the Taylor series above, we get
\begin{equation*}
\f_n(\pi_n(\p)) = \f_n(\pi_{n-1}(\p)) + D\f(\p_0)\p_n,\quad \forall n\geq 1,
\end{equation*}
and there is no contribution from the higher order terms since they have only coefficients of degree $2n$ or more. Introducing the $6\times 6$ matrices $M_n=n\lambda I_6-\txtD f(X_0)$, which are invertible for all $n\geq 2$, the invariance equation~\eqref{eq:invariance_coeff} for the coefficients $\p_n$, $n\geq 2$, rewrites
\begin{align*}
n\lambda \p_n - \txtD f(X_0)\p_n &= \f_n(\pi_{n-1}(\p)) \\
\p_n &= M_n^{-1}\f_n(\pi_{n-1}(\p)).
\end{align*}
Therefore, starting from $\p_0=X_0$ and $\p_1=\gamma v$, the coefficients $\p_n$ can be computed recursively. However, in practice one can only compute finitely many coefficients. Assume that the coefficients $\p_2,\ldots,\p_{N-1}$ have been computed, for a given $N\in\N$. For any $\p\in\X_\eta$, define $\hat \p:= (\p_0,\ldots,\p_{N-1})$, $\check \p:=(\p_N,\p_{N+1},\ldots)$. If $\gamma$ is chosen appropriately (see Remark~\ref{rem:gamma}), and if the truncation mode $N$ is large enough, the finite series given by $\hat p$ should be a good approximate parameterization of the local unstable manifold. To verify this claim, we introduce the operator
\begin{equation}
\label{def:T_manifold}
\T:\begin{cases}
\X_\eta\to \X_\eta \\
\check \p \mapsto \T(\check \p)=\left(\T_n(\check \p)\right)_{n\geq N},
\end{cases}
\end{equation}
where
\begin{equation*}
\T_n(\check \p) = M_n^{-1}\f_n(\pi_{n-1}(\p)),\quad \forall n\geq N,
\end{equation*}
with $\p=(\hat \p,\check \p)$. For convenience we also introduce the function $\g$ defined by $\g_n(\p)=\f_n(\pi_{n-1}(\p))$, so that $\T_n(\check \p) = M_n^{-1}\g_n(\p)$. We note that the coefficients of $\hat\p$ are now simply parameters that have already been computed, rather than variables. If $\check\p$ is a fixed point of $\T$, then $\p=(\hat\p,\check\p)$ solves~\eqref{eq:invariance_coeff} and hence corresponds to an exact parameterization of the local unstable manifold. Our goal is to show that there exists a fixed point of $\T$ in a neighborhood of $\0$, which would prove that $\hat\p$ is indeed an approximate parameterization. Our main tool is going to be the following corollary of Theorem~\ref{th:T}.

\begin{corollary}
\label{cor:T}
Let $\XX$ be a Banach space, and $T:\XX\to\XX$ a polynomial map of degree 4. Let $\barX\in\XX$ and $Y$, $Z_1$, $Z_2$, $Z_3$, $Z_4$ be non negative constants such that
\begin{subequations}

\label{def:bounds_corT}
\begin{align}
\label{def:Y_corT}
\left\Vert T(\barX)-\barX\right\Vert &\leq Y \\
\label{def:Z1_corT}
\left\Vert \txtD T(\barX)\right\Vert &\leq Z_1 \\
\label{def:Z2_corT}
\left\Vert \txtD^2T(\barX)\right\Vert &\leq Z_2 \\
\label{def:Z3_corT}
\left\Vert \txtD^3T(\barX)\right\Vert &\leq Z_3 \\
\label{def:Z4_corT}
\left\Vert \txtD ^4T(\barX)\right\Vert &\leq Z_4.
\end{align}
\end{subequations}
Consider
\begin{subequations}
\label{def:radii_pol_T}
\begin{align}
\label{def:radii_pol1_T}
P(r)&=\frac{Z_4}{24}r^4+\frac{Z_3}{6}r^3+\frac{Z_2}{2}r^2-(1-Z_1)r+Y \\
Q(r)&=\frac{Z_4}{6}r^3+\frac{Z_3}{2}r^2+Z_2r-(1-Z_1)
\end{align}
\end{subequations}
If there exists $\br>0$ such that $P(\br)<0$, then for all $r$ in the non empty interval $(r_{min},r_{max})$, $T$ has a unique fixed point in $\ball(\barX,r)$, where $r_{min}$ is the smallest positive root of $P$ and $r_{max}$ is the unique positive root of $Q$.
\end{corollary}
\begin{proof}
Just notice that $Z$ defined by
\begin{equation*}
Z(r)=Z_1+Z_2r+\frac{Z_3}{2}r^2+\frac{Z_4}{6}r^3
\end{equation*}
satisfies~\eqref{def:Z_T}. Condition~\eqref{cond:rY_T} is then equivalent to having $P(\br)<0$, while condition~\eqref{cond:rZ_T} is equivalent to having $Q(\br)<0$, and Theorem~\ref{th:T} can then be applied to conclude.
\end{proof}

\begin{remark}
\label{rem:non_pol}
At this stage, having that $T$ is a polynomial is convenient but absolutely not mandatory. If $T$ does not have a finite Taylor expansion (or if it is too cumbersome to use the full Taylor expansion), one can proceed is in Section~\ref{sec:background_validation} and consider an a priori radius $r^*>0$, together with non negative constants $Y$, $Z_1$ and $Z_2$ such that
\begin{subequations}
\begin{align*}
\left\Vert T(\barX)-\barX\right\Vert &\leq Y \\
\left\Vert \txtD T(\barX)\right\Vert &\leq Z_1 \\
\left\Vert \txtD T(X)- \txtD T(\barX) \right\Vert &\leq Z_2 \left\Vert X-\barX \right\Vert,\quad \forall~X\in \ball(\barX,r^*).
\end{align*}
\end{subequations}
We point out that, in practice, such a $Z_2$ can be obtained by bounding \begin{equation*}
\sup\limits_{X\in \ball(\barX,r^*)} \left\Vert \txtD ^2T(X) \right\Vert.
\end{equation*}
Then, defining 
\begin{equation*}
Z(r)=Z_1+Z_2r,
\end{equation*}
condition~\eqref{def:Z_T} holds for all $X\in \ball(\barX,r^*)$. Therefore, if assumptions~\eqref{cond:r_T} are satisfied for some $\br\leq r^*$, we can still conclude that $T$ has a unique fixed point in $\ball(\barX,\br)$.
\end{remark}

\begin{remark}
\label{rem:radii_pol}
$P$ and $Q$ are sometimes called \emph{radii polynomials} in the literature.
\end{remark}

Our goal is to apply this corollary to the operator $\T$ defined just above, with $\barX=\0$, to validate the approximate parameterization defined by $\hat \p$.

\subsection{The bounds needed for the validation}

In this subsection, we obtain computable bounds $Y$ and $Z_i$, $i=1,\ldots,4$, satisfying~\eqref{def:bounds_corT}.

\subsubsection{The bound $Y$}

\begin{proposition}
\label{prop_Y_manifold}
Consider $\T$ defined in~\eqref{def:T_manifold}, $\eta\in\R^6_{>0}$ and let
\begin{equation}
\label{def:Y_manifold}
Y= \sum_{n=N}^{4N-4} \sum_{i=1}^6 \vert \T_n^{(i)}(\0)\vert \eta^{(i)}.
\end{equation}
Then 
\begin{equation*}
\left\Vert \T(\0) \right\Vert_{\X_\eta} \leq Y.
\end{equation*}
\end{proposition}
\begin{proof}
Since $f$ is quartic and $\hat\p$ only has coefficients up to order $N-1$, $\T_n(\0)=0$ for all $n> 4N-4$.
\end{proof}

Notice that $Y$ can be rigorously evaluated on a computer (or more precisely upper-bounded, using interval arithmetic).

\subsubsection{The bound $Z_1$}
\label{sec:Z1_manifold}

First, we need to control the norm of $M_n^{-1}$.

\begin{lemma}
\label{lem:norm_Mn_inv1}
Let $n\in\N$ and $\eta\in\R_{>0}^6$ such that the following holds
\begin{equation*}
n\lambda > \left\vert \txtD f(X_0) \right\vert_\eta.
\end{equation*}
Then we can conclude that
\begin{equation*}
\left\vert M_n^{-1} \right\vert_\eta \leq \frac{1}{n\lambda-\left\vert \txtD f(X_0) \right\vert_\eta}.
\end{equation*}
\end{lemma}
\begin{proof}
This follows directly from a standard Neumann series argument.
\end{proof}

\begin{remark}
\label{rem:eta_manifold}
In practice, the bound on $\left\vert M_n^{-1} \right\vert_\eta$ ends up being one of the critical factors influencing the success or failure of the validation procedure. Therefore, we choose $\eta$ so as to minimize $\left\vert \txtD f(X_0) \right\vert_\eta$ (see Proposition~\ref{prop:frob}), in order to get the smallest possible estimate out of Lemma~\ref{lem:norm_Mn_inv1}.
\end{remark}

The downside of Lemma~\ref{lem:norm_Mn_inv1} is that it can only be used for $n$ larger than $\frac{\left\vert \txtD f(X_0) \right\vert_\eta}{\lambda}$. If one wishes to choose the truncation level $N$ smaller than this threshold, the following alternative bound can be used for the low order modes, i.e., for all modes below the truncation level.

\begin{lemma}
\label{lem:norm_Mn_inv2}
Assume there exists a $6\times 6$ matrix $Q$ such that
\begin{equation*}
\txtD f(X_0)=Q^{-1} \Lambda Q,
\end{equation*}
where
\begin{equation*}
\Lambda=\begin{pmatrix}
\lambda & & & & &  \\
 & \mu & & & &  \\
 & & 0 & & &  \\  
 & & & 0 & &  \\ 
 & & & & 0 &  \\ 
 & & & & & 0 \\  
\end{pmatrix}.
\end{equation*}
Then, for all $\eta\in\R_{>0}^6$ and $n\geq 2$,
\begin{equation*}
\left\vert M_n^{-1} \right\vert_\eta \leq \frac{\left\vert Q^{-1} \right\vert_\eta \left\vert Q \right\vert_\eta}{(n-1)\lambda}.
\end{equation*}
\end{lemma}
\begin{proof}
Just notice that
\begin{equation*}
M_n^{-1}=Q^{-1}\begin{pmatrix}
\frac{1}{(n-1)\lambda} & & & & &  \\
 & \frac{1}{n\lambda-\mu} & & & &  \\
 & & \frac{1}{n\lambda} & & &  \\  
 & & & \frac{1}{n\lambda} & &  \\ 
 & & & & \frac{1}{n\lambda} &  \\ 
 & & & & & \frac{1}{n\lambda} \\  
\end{pmatrix}Q,
\end{equation*}
with $\mu<0$.
\end{proof}
\begin{remark}
In practice, such matrix $Q$ can be obtained numerically and then validated using the function \texttt{verifyeig} from \textsc{Intlab}, combined with the fact that we know a priori by Lemma~\ref{lem:extended_spectrum} that the eigenvalues of $\txtD f(X_0)$ can only be $\lambda$, $\mu$ and $0$.
\end{remark}

Combining the two above lemmas, we define
\begin{equation*}
\mathfrak{M}_n = \left\{\begin{aligned}
&\frac{\left\vert Q^{-1} \right\vert_\eta \left\vert Q \right\vert_\eta}{(n-1)\lambda},\qquad &n\leq\frac{\left\vert \txtD f(X_0) \right\vert_\eta}{\lambda}\\
&\min\left(\frac{1}{n\lambda-\left\vert \txtD f(X_0) \right\vert_\eta},\frac{\left\vert Q^{-1} \right\vert_\eta \left\vert Q \right\vert_\eta}{(n-1)\lambda}\right),\qquad &n>\frac{\left\vert \txtD f(X_0) \right\vert_\eta}{\lambda} 
\end{aligned}\right.
\end{equation*}
%
Next, we introducing the linear operator $\tilde{\M}:\X_\eta\to\X_\eta$
\begin{equation*}
\tilde{\M} = \begin{pmatrix}
M_N^{-1} & & \\
 & M_{N+1}^{-1} &  \\
 & & \ddots  
\end{pmatrix}.
\end{equation*}
We are now ready to define the $Z_1$ bound.

\begin{proposition}
\label{prop:Z1_manifold}
Consider $\T$ defined in~\eqref{def:T_manifold}, $\eta\in\R^6_{>0}$ and let
\begin{align}
\label{def:Z1_manifold}
Z_1 &= \mathfrak{M}_N \left\vert \left\Vert \txtD \g(\hat \p)\right\Vert_1 \right\vert_\eta \nonumber\\
&= \mathfrak{M}_N \max_{1\leq j\leq 6} \frac{1}{\eta^{(j)}}\sum_{i=1}^6 \left\Vert \txtD _j\g^{(i)}(\hat \p) \right\Vert_1. 
\end{align}
Then we get
\begin{equation*}
\left\Vert \txtD \T(\0) \right\Vert_{\X_\eta} \leq Z_1.
\end{equation*}
\end{proposition}
\begin{proof}
Using that $\T=\tilde{\M}\g$, Lemma~\ref{lem:norm_Mn_inv1}, Lemma~\ref{lem:norm_Mn_inv2} and then~\eqref{eq:lin_op_norm_computable}, we estimate
\begin{align*}
\left\Vert \txtD \T(\0) \right\Vert_{\X_\eta} &\leq \left\Vert \tilde{\M} \right\Vert_{\X_\eta} \left\Vert \txtD \g(\hat \p)\right\Vert_{\X_\eta} \\
&= \sup_{n\geq N} \left\vert M_n^{-1} \right\vert_\eta \left\Vert \txtD \g(\hat \p)\right\Vert_{\X_\eta} \\
&\leq \mathfrak{M}_N \left\Vert \txtD \g(\hat \p)\right\Vert_{\X_\eta} \\
&\leq \mathfrak{M}_N \left\vert \left\Vert \txtD \g(\hat \p)\right\Vert_1 \right\vert_\eta,
\end{align*}
which finishes the proof.
\end{proof}

We emphasize that the bound $Z_1$ defined just above is computable, since the $\ell^1$ operator norms $\left\Vert \txtD _j\g^{(i)}(\hat \p)\right\Vert_{1}$ are very easy to evaluate. Indeed, the action of the linear operator $\txtD _j\f^{(i)}(\p)$ on $\ell^1$ is nothing but a convolution with the sequence representing the function $\txtD _jf^{(i)}(p)$, and therefore these $\ell^1$ operator norms are simply given by the $\ell^1$ norm of the corresponding vector. For instance, since
\begin{equation*}
\txtD _2f^{(1)}(X)=-\sum_{i=1}^4b^{(i)}X^{(i+2)},
\end{equation*}
we have
\begin{equation*}
\left\Vert \txtD _2\f^{(1)}(\p) \right\Vert_1= \left\Vert \sum_{i=1}^4b^{(i)}\p^{(i+2)}\right\Vert_1.
\end{equation*}
To go back to $\txtD \g(\hat\p)$, notice that, for any $\bu,\bv\in\ell^1$,
\begin{align*}
\left(\bu\star \pi_{n-1} (\bv)\right)_n &= \sum_{k=1}^n \bu_k \bv_{n-k} \\
& = \left(\pi^0(\bu)\star \bv\right)_n,
\end{align*}
where $\pi^0(\bu):=(0,\bu_1,\bu_2,\ldots)$. Finally, using that $\pi_{n-1}(\hat\p)=\hat\p$ for all $n\geq N$, we get that
\begin{equation*}
\txtD \g_n(\hat\p)=\txtD \f_n(\hat\p)\pi_{n-1},
\end{equation*}
and thus the $\ell^1$ operator norm of $\txtD _j\g^{(i)}(\hat \p)$ is nothing but the $\ell^1$ norm of $\pi^0 \bu$, where $\bu$ is the vector representing $\txtD _jf^{(i)}(\hat p)$. For instance, we have 
\begin{equation*}
\left\Vert \txtD _2\g^{(1)}(\hat \p) \right\Vert_1= \left\Vert \pi^0\left(\sum_{i=1}^4b^{(i)}\hat \p^{(i+2)}\right)\right\Vert_1
\end{equation*}
and since $\hat \p$ only has finitely many non zero coefficients, this quantity can be evaluated (or more precisely, upper-bounded using interval arithmetic).

\begin{remark}
\label{rem:gamma}
\textbf{About the role of $\gamma$.} The approximate parameterization $\hat \p$ is uniquely determined by the choice of the scaling $\gamma$ in~\eqref{eq:invariance_coeff}. In this remark, we note $\hat\p=\hat\p(\gamma)$ to highlight this dependency. One crucial observation is that we have
\begin{equation*}
\hat\p_n(\gamma) = \gamma^n \hat\p_n(1),\quad \forall~n\in\N.
\end{equation*}
Besides, denoting by $\bu=\bu(\gamma)$ the vector representing $\txtD _jf^{(i)}(\hat p(\gamma))$, we also have that
\begin{equation*}
\bu_n(\gamma) = \gamma^n \bu_n(1),\quad \forall~n\in\N.
\end{equation*}
Hence, 
\begin{align*}
\left\Vert \txtD _j\g^{(i)}(\hat \p(\gamma)) \right\Vert_1 &= \left\Vert \pi^0(\bu(\gamma))\right\Vert_1 \\
&= \sum_{n= 1}^{3N-2} \left\vert \bu_n(\gamma)\right\vert \\
&= \sum_{n= 1}^{3N-2} \vert\gamma\vert^n \left\vert \bu_n(1)\right\vert.
\end{align*}
Therefore, $\left\Vert \txtD _j\g^{(i)}(\hat \p(\gamma)) \right\Vert_1$ goes to $0$ when $\gamma$ goes to $0$, which means that we can have a $Z_1$ bound that is arbitrarily small (and in particular strictly less than $1$) by taking $\gamma$ small enough. In other words, up to considering a small enough patch of the local manifold, we can always get a contraction.
\end{remark}

\subsubsection{The bounds $Z_k$, $2\leq k\leq 4$}

\begin{proposition}
\label{prop:Zk_manifold}
Consider $\T$ defined in~\eqref{def:T_manifold}, $\eta\in\R^6_{>0}$ and define for $2\leq k\leq 4$
\begin{align}
\label{def:Zk_manifold}
Z_k = \mathfrak{M}_N \max_{1\leq j_1,\ldots,j_k \leq 6} \frac{1}{\eta^{(j_1)}\ldots\eta^{(j_k)}}\sum_{i=1}^6 \left\Vert \txtD ^k_{(j_1,\ldots,j_k)}\g^{(i)}(\hat \p) \right\Vert_1. 
\end{align}
Then
\begin{equation*}
\left\Vert \txtD ^k\T(\0) \right\Vert_{\X_\eta} \leq Z_k.
\end{equation*}
\end{proposition}
\begin{proof}
We have 
\begin{align*}
\left\Vert \txtD ^k\T(\0)\right\Vert_{\X_\eta} \leq \mathfrak{M}_N \left\Vert \txtD ^k\g(\hat \p)\right\Vert_{\X_\eta},
\end{align*}
and we estimate the norm of the $k$-linear operator $\txtD ^k\g(\hat \p)$ as in~\eqref{eq:multilin_op_norm_computable}.
\end{proof}

Again, this bound can be evaluated since we can compute the $\ell^1$ operator norms
\begin{equation*}
\left\Vert \txtD ^k_{(j_1,\ldots,j_k)}\g^{(i)}(\hat \p)\right\Vert_1.
\end{equation*}
Indeed, such multilinear operator norm is equal to the $\ell^1$ norm of the vector of Taylor coefficients representing
\begin{equation*}
\txtD ^k_{(j_1,\ldots,j_k)}g^{(i)}(\hat p),
\end{equation*}
and thus can be computed using interval arithmetic.
For instance
\begin{equation*}
\left\Vert \txtD ^2_{(2,3)}\g^{(1)}(\hat \p)\right\Vert_1=\left\vert b^{(1)}\right\vert,
\end{equation*}
and
\begin{equation*}
\left\Vert \txtD ^2_{(1,1)}\g^{(3)}(\hat \p)\right\Vert_1= \left\Vert 2\sum_{j=1}^4\left( (4a^{(1)}a^{(j)}+b^{(1)}b^{(j)})\hat\p^{(3)}\ast\hat\p^{(j+2)}\right)\right\Vert_1.
\end{equation*}

\subsection{Summary}
\label{sec:summary_manifold}

We recall that we explained in Sections~\ref{sec:step_zeros} and~\ref{sec:step_manifold} how to obtain a rigorous enclosure of saddles points and of the associated eigenvalues and eigenvectors. We now show how all the estimates derived up to now can be combined with Corollary~\ref{cor:T} to rigorously validate a local stable manifold around each saddle.

\begin{theorem}
\label{th:recap_manifold}
Let $(x_0,y_0)\in\R^2$ be saddle point of $V$ and consider $X_0$ defined as in~\eqref{def:X} with $z_0=\psi(x_0,y_0)$. Denote by $\lambda>0$ the unstable eigenvalue of $\txtD f(X_0)$ and by $v$ an associated eigenvector. Let $\gamma>0$, $N\in\N_{\geq 2}$ and $\eta\in\R^6_{>0}$.  Assume that 
\begin{equation}
\hat \p:= (\p_0,\ldots,\p_{N-1}) 
\end{equation}
has been computed recursively rigorously (i.e. with interval arithmetic) so that $\p_n$ solves the invariance equation~\eqref{eq:invariance_coeff} for all $n<N$. Consider the operator $\T$ defined in~\eqref{def:T_manifold} and the bounds $Y$, $Z_1$ and $Z_k$, $2\leq k\leq 4$ defined in~\eqref{def:Y_manifold}, \eqref{def:Z1_manifold} and \eqref{def:Zk_manifold} respectively. Finally, assume there exists $\br>0$ such that $P(\br)<0$ and $Q(\br)<0$, with $P$ and $Q$ defined in~\eqref{def:radii_pol_T}. Then there exists $\p_n\in\R^6$, $n\geq N$, such that
\begin{equation*}
\sum_{i=1}^6\sum_{n=N}^\infty \vert \p_n^{(i)}\vert \eta^{(i)} \leq \br,
\end{equation*}
and such that the function $p$ defined as in~\eqref{def:p} satisfies the invariance equation~\eqref{eq:invariance}, that is $p([-1,1])\subset \R^6$ is a local unstable manifold of $X_0$ for the vector field $f$. Besides
\begin{equation*}
\left\{ \left(p^{(1)}(\theta),p^{(2)}(\theta)\right),\ \theta\in [-1,1] \right\}\subset \R^2
\end{equation*}
is a local unstable manifold of $(x_0,y_0)$ for the vector field $-\nabla V$.
\end{theorem}

This theorem proves that the function $\hat p$ defined by
\begin{equation*}
\hat p(\theta) =\sum_{n=0}^{N-1} \p_n \theta^n
\end{equation*}
is an approximate parameterization of the local unstable manifold of $X_0$ in the sense that
\begin{equation*}
\left\vert p(\theta)-\hat p(\theta) \right\vert_\eta \leq \br ,\quad \forall~\theta\in[-1,1].
\end{equation*}
We give in Section~\ref{sec:results} several examples (with explicit values of the parameters) of 
applications of Theorem~\ref{th:recap_manifold} to validate local manifolds for the M\"uller-Brown potential.

\section{Validation of connecting orbits}
\label{sec:connection}

\subsection{Background}
\label{sec:background_orbit}

The material introduced in this subsection is standard, and mainly included for the sake of completeness and to fix some notations. The reader familiar with the notions presented here might skip Section~\ref{sec:background_orbit} on first reading, and only refer back when needed.

\subsubsection{Notations}

For $M\in\N_{\geq 1}$, we consider sequences $\bX\in \left(\left(\R^\N\right)^{6}\right)^M\simeq \left(\R^\N\right)^{6\times M} \simeq \left(\R^{6\times M}\right)^\N$. We thus write
\begin{equation*}
\bX=\left(\bX^{(1)},\ldots,\bX^{(M)}\right),\quad\text{where}\quad \bX^{(m)}=\left(\bX^{(m,1)},\ldots,\bX^{(m,6)}\right) \in \left(\R^\N\right)^6.
\end{equation*}
For $k\in\N$, we then denote by $T_k$ the Chebyshev polynomial of order $k$, defined by $T_k(\cos(\theta))=\cos(k\theta)$. We consider $\tau>0$ and a partition $0=t^{(0)}<t^{(1)}<\ldots<t^{(M)}=\tau$. For all $k\in\N$ and $m=1,\ldots,M$, we also introduce the \emph{rescaled} Chebyshev polynomial $T^{(m)}_k$, defined as
\begin{equation*}
T^{(m)}_k(t)=T_k\left(\frac{2t-t^{(m)}-t^{(m-1)}}{t^{(m)}-t^{(m-1)}}\right).
\end{equation*}
We associate to a sequence $\bX\in \left(\left(\R^\N\right)^{6}\right)^M$ the function $X:[0,\tau]\to \R^6$ defined as a piece-wise Chebyshev series by
\begin{equation}
\label{eq:piecewise_cheb}
X(t) = \bX^{(m)}_0+2\sum_{k=1}^\infty \bX^{(m)}_k T^{(m)}_k(t), \quad \forall~t\in(t^{(m-1)},t^{(m)}),\ \forall~m\in\{1,\ldots,M\}.
\end{equation}
Notice that, for all $k\in\N$, $\bX^{(m)}_k=\left(\bX^{(m,1)}_k,\ldots,\bX^{(m,6)}_k\right)\in \R^6$. We also introduce in the next subsection a subspace of $\left(\left(\R^6\right)^\N\right)^M$ for which such series is guaranteed to converge. Given two sequences $\bu,\bv\in\R^N$, we denote by $\bu\ast\bv$ their convolution product
\begin{equation*}
\left(\bu\ast\bv\right)_k =\sum_{l\in\Z} \bu_{\vert l\vert} \bv_{\vert k-l\vert}.
\end{equation*}
In this Section, $\f:\left(\R^\N\right)^6 \to \left(\R^\N\right)^6$ denotes the map such that, for any function $X$ represented by a piece-wise Chebyshev series $\bX$, $\f(\bX)$ are the coefficients of the piece-wise Chebyshev series representation of $f(X)$, that is 
\begin{equation*}
\f^{(m,1)}_k(\bX)=  -\sum_{i=1}^4\left(2a^{(i)}\left(\bX^{(m,1)}\ast \bX^{(m,i+2)}\right)_k+b^{(i)}\left(\bX^{(m,2)}\ast \bX^{(m,i+2)}\right)_k-w^{(i)}_1\bX^{(m,i+2)}_k\right),
\end{equation*}
and so on.

\subsubsection{Norms and Banach spaces}

\begin{definition}
For $\nu>1$, we define the following weighted $\ell^1$ norm. We first introduce the weights
\begin{equation*}
\xi_k(\nu) = \left\{\begin{aligned}
&1 &\quad k=0, \\
&2\nu^k &\quad k\geq 1.
\end{aligned}\right.
\end{equation*} 
For $\bu\in\R^\N$, we then define
\begin{align*}
\left\Vert \bu \right\Vert_\nu &= \sum_{k=0}^\infty \vert \bu_k \vert \xi_k(\nu) \\
&= \vert \bu_0 \vert + 2\sum_{k=1}^\infty \vert \bu_k \vert \nu^k.
\end{align*}
\end{definition}
We recall that we have $\left\Vert \bu\ast\bv \right\Vert_\nu \leq \left\Vert \bu 
\right\Vert_\nu \left\Vert \bv \right\Vert_\nu$.

\begin{definition}
Let $\nu>1$ and $\eta\in \R_{>0}^{6M}$. For $\bu\in\left(\R^\N\right)^6$ we define
\begin{align*}
\left\Vert \bu \right\Vert_{\X_{\eta^{(m)},\nu}} &= \sum_{i=1}^6 \left(\vert \bu^{(i)}_0 \vert + 2\sum_{k=1}^\infty \vert \bu^{(i)}_k \vert \nu^k\right)\eta^{(m,i)} \\
&= \sum_{i=1}^6 \left\Vert \bu^{(i)} \right\Vert_\nu\eta^{(m,i)} \\
&= \left\vert \left\Vert \bu \right\Vert_\nu \right\vert_{\eta^{(m)}},
\end{align*}
with again a slight abuse of notation on the last line, $\left\Vert \bu \right\Vert_\nu$ being understood as $\left(\left\Vert \bu^{(i)} \right\Vert_\nu\right)_{i=1,\ldots,6}$, and for $X\in\R^6$,
\begin{equation*}
\left\vert X \right\vert_{\eta^{(m)}} = \sum_{i=1}^6 \vert X^{(i)} \vert \eta^{(m,i)}.
\end{equation*} 
We then introduce
\begin{equation*}
\X_{\eta^{(m)},\nu}=\left\{ \bu\in \left(\R^\N\right)^{6} ,\ \left\Vert \bu\right\Vert_{\X_{\eta^{(m)},\nu}}<\infty \right\}.
\end{equation*} 
Next, we consider the product space $\X_{\eta,\nu}=\prod_{m=1}^M \X_{\eta^{(m)},\nu}$ endowed with the supremum norm. That is, for $\bX\in \left(\R^\N\right)^{6M}$ we define
\begin{align*}
\left\Vert \bX \right\Vert_{\X_{\eta,\nu}} &= \max_{1\leq m\leq M} \left\Vert \bX^{(m)} \right\Vert_{\X_{\eta^{(m)},\nu}} \\
&= \max_{1\leq m\leq M} \sum_{i=1}^6 \left(\vert \bX^{(m,i)}_0 \vert + 2\sum_{k=1}^\infty \vert \bX^{(m,i)}_k \vert \nu^k\right)\eta^{(m,i)},
\end{align*}
and
\begin{equation*}
\X_{\eta,\nu}=\left\{ \bX\in \left(\R^\N\right)^{6M} ,\ \left\Vert \bX\right\Vert_{\X_{\eta,\nu}}<\infty \right\}.
\end{equation*}
\end{definition}

\subsubsection{Operator norms}

We use the same block-representation as in Section~\ref{sec:op_blocks_manifold}, with one extra layer. We write a linear operator $\A$ on $\X_{\eta,\nu}$ as 
\begin{equation*}
\A=\begin{pmatrix}
\A^{(1;1)} & \ldots & \A^{(1;M)} \\
\vdots & \ddots & \vdots \\ 
\A^{(M;1)} & \ldots & \A^{(M;M)} \\  
\end{pmatrix},
\end{equation*}
where $\A^{(m;n)}$ is a linear operator from $\X_{\eta^{(n)},\nu}$ to $\X_{\eta^{(m)},\nu}$ and can be written himself in block-form
\begin{equation*}
\A^{(m;n)}=\begin{pmatrix}
\A^{(m,1;n,1)} & \ldots & \A^{(m,1;n,6)} \\
\vdots & \ddots & \vdots \\ 
\A^{(m,6;n,1)} & \ldots & \A^{(m,6;n,6)} \\  
\end{pmatrix},
\end{equation*}
each $\A^{(m,i;n,j)}$ being a linear operator on $\ell^1_\nu$. We also write
\begin{equation*}
\A^{(m)}=\begin{pmatrix}
\A^{(m;1)} & \ldots & \A^{(m;M)} 
\end{pmatrix},
\end{equation*}
the $m$-th "operator row" of $\A$. Notice that, for $\bX\in\X_{\eta,\nu}$, we have
\begin{align*}
\left(\A\bX\right)^{(m)} &= \A^{(m)}\bX \\
&= \sum_{n=1}^M \A^{(m;n)}\bX^{(n)}
\end{align*}
We slightly abuse the notation by applying the $\ell^1_\nu$ operator norm component wise to $\A^{(m;n)}$, that is
\begin{equation*}
\left\Vert \A^{(m;n)}\right\Vert_\nu=\begin{pmatrix}
\left\Vert\A^{(m,1;n,1)}\right\Vert_\nu & \ldots & \left\Vert\A^{(m,1;n,6)}\right\Vert_\nu \\
\vdots & \ddots & \vdots \\ 
\left\Vert\A^{(m,6;n,1)}\right\Vert_\nu & \ldots & \left\Vert\A^{(m,6;n,6)}\right\Vert_\nu \\  
\end{pmatrix}.
\end{equation*}
We recall that
\begin{align}
\label{eq:ell_1_nu_op_norm}
\left\Vert \A^{(m,i;n,j)} \right\Vert_\nu = \sup_{l\in\N} \frac{1}{\xi_l(\nu)} \sum_{k\in\N} \left\vert \A^{(m,i;n,j)}_{k,l} \right\vert \xi_k(\nu).
\end{align}
We also introduce a notation for weighted $1$-operator norms with two different sets of weights: 
\begin{equation*}
\left\vert \left\Vert \A^{(m;n)}\right\Vert_\nu \right\vert_{\eta^{(n)}\to\eta^{(m)}} = \max_{1\leq j\leq 6} \frac{1}{\eta^{(n,j)}} \sum_{i=1}^6 \left\Vert \A^{(m,i;n,j)} \right\Vert_\nu \eta^{(m,i)}
\end{equation*}
We then have
\begin{align}
\label{eq:lin_op_norm_eta_nu}
\left\Vert \A^{(m)} \right\Vert_{\X_{\eta,\nu}\to\X_{\eta^{(m)},\nu}} & \leq \sum_{n=1}^M \left\Vert \A^{(m;n)} \right\Vert_{\X_{\eta^{(n)},\nu}\to\X_{\eta^{(m)},\nu}} \nonumber\\
& \leq \sum_{n=1}^M \left\vert \left\Vert \A^{(m;n)}\right\Vert_\nu \right\vert_{\eta^{(n)}\to\eta^{(m)}}.
\end{align}

\subsection{Framework}
\label{sec:framework_orbit}

We want $X$ to satisfy
\begin{equation*}
X'=f(X),\quad X(0)=p(1),\quad X(\tau) \in W^{\textnormal{s}}_{\textnormal{loc}}(M).
\end{equation*}
If we look for a piece-wise Chebyshev series representation of $X$ as in~\eqref{eq:piecewise_cheb}, we obtain the following equations on the coefficients $\bX$
\begin{equation*}
\left\{\begin{aligned}
&k\bX^{(m)}_k=\frac{t^{(m)}-t^{(m-1)}}{4}\left(\f^{(m)}_{k-1}(\bX)-\f^{(m)}_{k+1}(\bX)\right),\quad \forall~k\geq 1,\ \forall~1\leq m\leq M \\
&\bX^{(1)}_0+2\sum_{k=1}^\infty (-1)^k\bX^{(1)}_k=p(1), \\
&\bX^{(m)}_0+2\sum_{k=1}^\infty \bX^{(m)}_k = \bX^{(m+1)}_{0}+2\sum_{k=1}^\infty (-1)^k\bX^{(m+1)}_{k},\quad \forall~1\leq m\leq M-1 \\
&\bX^{(M)}_0+2\sum_{k=1}^\infty \bX^{(M)}_k \in W^{\textnormal{s}}_{\textnormal{loc}}(M).
\end{aligned}\right.
\end{equation*}
The derivation of these equations follows readily (see for instance~\cite{LesRei14,BerShe15}) from well known properties of the Chebyshev polynomials, namely
\begin{equation*}
T_k(1)=1,\quad T_k(-1)=(-1)^k\quad \text{and} \quad \int T_k =\frac{1}{2}\left(\frac{T_{k+1}}{k+1}-\frac{T_{k-1}}{k-1}\right).
\end{equation*}
We define $\F$ by
\begin{align}
\label{def:F_orbit}
\F^{(m)}_0(\bX) &= \bX^{(m)}_0+2\sum_{k=1}^\infty (-1)^k\bX^{(m)}_k-\left(\bX^{(m-1)}_0+2\sum_{k=1}^\infty \bX^{(m-1)}_k\right),\quad &\forall~1\leq m\leq M \nonumber\\
\F^{(m)}_k(\bX) &= k\bX^{(m)}_k-\frac{t^{(m)}-t^{(m-1)}}{4}\left(\f^{(m)}_{k-1}(\bX)-\f^{(m)}_{k+1}(\bX)\right) ,\quad &\forall~1\leq m\leq M,\ \forall~k\geq 1,
\end{align}
with the convention $\bX^{(0)}_0+2\sum_{k=1}^\infty \bX^{(0)}_k=p(1)$, and we want to validate a zero $\bX$ of $\F$ in $\X_{\eta,\nu}$. The main difference with the situation considered in Section~\ref{sec:manifold} is that we cannot solve first for a finite number of modes and then get a contraction for the tail, because the system is now fully coupled (since we use Chebyshev series) instead of triangular (since we used Taylor series for the manifold).  To be precise, one can still solve \emph{numerically} for a finite number of modes by considering a truncated system, but the obtained numerical solution and the tail must both be validated simultaneously. To do so, we introduce a Newton-like reformulation, based on a truncated system, that is suitable for our a posteriori validation procedure. We denote by $\pi_K:\X_{\eta,\nu}\to \R^{6MK}$ the finite dimensional projection obtained by truncating the Chebyshev modes of order $K$ and higher, that is
\begin{equation*}
\pi_K(\bX)=(\bX_0,\ldots,\bX_{K-1}).
\end{equation*}
We also denote by $\imath_K$ the natural injection from $\R^{6MK}$ to $\X_{\eta,\nu}$. First, we compute an approximate solution $\bar\bX$ by solving numerically the finite dimensional problem $\F^{[K]}=\pi_K \circ \F \circ \imath_K =0$. We use the same notation to denote $\bar \bX\in\R^{6MK}$ and its injection in $\X_{\eta,\nu}$. Then, we define the linear operator $\A^\dag$ by
\begin{equation}
\label{def:Adag}
\left\{\begin{aligned}
&\A^\dag \pi_K \bX = \txtD \F^{[K]}(\bar\bX) \pi_K \bX, \\
&\A^\dag \bX_k = k \bX_k,\quad \forall~k\geq K.
\end{aligned}\right.
\end{equation}
Next, we compute $\A^{[K]}$ an approximate inverse of $\txtD \F^{[K]}(\bar\bX)$, and define the linear operator $\A$ by
\begin{equation}
\label{def:A}
\left\{\begin{aligned}
&\A \pi_K \bX = \A^{[K]} \pi_K \bX, \\
&\A \bX_k = \frac{1}{k}\bX_k,\quad \forall~k\geq K.
\end{aligned}\right.
\end{equation}
The operator $\A$ enables us to recover an equivalent fixed-point formulation which should give a contraction around $\bar \bX$, by considering the Newton-like operator
\begin{equation*}
I_{\X_{\eta,\nu}}-\A\F.
\end{equation*}
Our goal is to apply a variant of Theorem~\ref{th:T} to this Newton-like operator, to validate $\bar \bX$ by proving the existence of a true zero of $\F$ in a neighborhood of $\bar \bX$.

\begin{corollary}
\label{cor:F}
Let $M\in\N_{\geq 1}$, $\left(\XX^{(m)},\left\Vert\cdot\right\Vert_{\XX^{(m)}}\right)$ and  $\left(\YY^{(m)},\left\Vert\cdot\right\Vert_{\YY^{(m)}}\right)$ be Banach spaces, for all $1\leq m\leq M$. Consider the product spaces 
\begin{equation*}
\XX = \prod_{m=1}^M \XX^{(m)},\quad \text{and}\quad \YY = \prod_{m=1}^M \YY^{(m)}
\end{equation*}
endowed with the norms 
\begin{equation*}
\left\Vert\cdot\right\Vert_{\XX} = \max\limits_{1\leq m\leq M} \left\Vert\cdot\right\Vert_{\XX^{(m)}},\quad \text{and}\quad \left\Vert\cdot\right\Vert_{\YY} = \max\limits_{1\leq m\leq M} \left\Vert\cdot\right\Vert_{\YY^{(m)}}.
\end{equation*}
Consider $F=\left(F^{(m)}\right)_{1\leq m\leq M}:\XX\to\YY$ a polynomial map of degree 4, $\barX\in\XX$, $A^\dag:\XX\to\YY$ a linear operator, and $A:\YY\to\XX$ an injective linear operator. Finally, assume that, for all $1\leq m\leq M$, $Y^{(m)}$, $Z_0^{(m)}$, $Z_1^{(m)}$, $Z_2^{(m)}$, $Z_3^{(m)}$, $Z_4^{(m)}$ are non negative constants such that
\begin{subequations}
\label{def:bounds_corF}
\begin{align}
\label{def:Y_corF}
\left\Vert \left(AF(\barX)\right)^{(m)}\right\Vert_{\XX^{(m)}} &\leq Y^{(m)} \\
\label{def:Z0_corF}
\left\Vert \left(I-AA^\dag\right)^{(m)}\right\Vert_{\XX\to\XX^{(m)}} &\leq Z_0^{(m)} \\
\label{def:Z1_corF}
\left\Vert \left(A(\txtD F(\barX)-A^\dag)\right)^{(m)}\right\Vert_{\XX\to\XX^{(m)}} &\leq Z_1^{(m)} \\
\label{def:Z2_corF}
\left\Vert \left(A\txtD ^2F(\barX)\right)^{(m)}\right\Vert_{\XX^2\to\XX^{(m)}} &\leq Z_2^{(m)} \\
\label{def:Z3_corF}
\left\Vert \left(A\txtD ^3F(\barX)\right)^{(m)}\right\Vert_{\XX^3\to\XX^{(m)}} &\leq Z_3^{(m)} \\
\label{def:Z4_corF}
\left\Vert \left(A\txtD ^4F(\barX)\right)^{(m)}\right\Vert_{\XX^4\to\XX^{(m)}} &\leq Z_4^{(m)}.
\end{align}
\end{subequations}
Consider, for all $1\leq m\leq M$,
\begin{subequations}
\label{def:radii_pol_F}
\begin{align}
P^{(m)}(r) &=\frac{Z_4^{(m)}}{24}r^4+\frac{Z_3^{(m)}}{6}r^3+\frac{Z_2^{(m)}}{2}r^2-\left(1-Z_1^{(m)}-Z_0^{(m)}\right)r+Y^{(m)} \\
Q^{(m)}(r) &=\frac{Z_4^{(m)}}{6}r^3+\frac{Z_3^{(m)}}{2}r^2+\left(Z_2\right)^{(m)}r-\left(1-Z_1^{(m)}-Z_0^{(m)}\right).
\end{align}
\end{subequations}
Assume that, for all $1\leq m\leq M$, $P^{(m)}$ has a positive root, and denote by $r_{min}^{(m)}$ the smallest positive root of $P^{(m)}$ and by $r_{max}^{(m)}$ is the unique positive root of $Q^{(m)}$. Finally, define
\begin{equation*}
r_{min}=\max\limits_{1\leq m\leq M} r_{min}^{(m)} \quad\text{and}\quad r_{max}=\min\limits_{1\leq m\leq M} r_{max}^{(m)}.
\end{equation*}
If $r_{min}<r_{max}$, then for all $r$ in $(r_{min},r_{max})$, $F$ has a unique zero in $\ball(\barX,r)$.
\end{corollary}
\begin{proof}
We consider the fixed point operator $T=I-AF$. Defining
\begin{equation*}
Z^{(m)}(r)=Z_0^{(m)} + Z_1^{(m)}+Z_2^{(m)}r+\frac{Z_3^{(m)}}{2}r^2+\frac{Z_4^{(m)}}{6}r^3,
\end{equation*}
and following the proof of Theorem~\ref{th:T}, component-wise, shows that $T$ is a contraction on $\ball(\barX,r)$, for all $r$ in $(r_{min},r_{max})$. This concludes the proof, since fixed points of $T$ correspond to zeros of $F$ by injectivity of $A$.
\end{proof}

Remarks~\ref{rem:non_pol} and~\ref{rem:radii_pol} also apply here. We now want to use this corollary for $\F$, $\A^\dag$ and $\A$ defined just above, to validate the approximate orbit defined by $\bar \bX$.

\subsection{The bounds needed for the validation}

In this subsection, we obtain computable bounds $Y^{(m)}$ and $Z_i^{(m)}$, $i=1,\ldots,4$, $m=1,\ldots,M$, satisfying~\eqref{def:bounds_corF}.

\subsubsection{The bound $Y$}

Since $f$ is polynomial and $\bX$ only has finitely many non zero coefficients, so does $\F(\bX)$, and hence 
\begin{equation}
\label{def:Y_orbit}
Y^{(m)} = \left\Vert \left(\A\F(\bX)\right)^{(m)}\right\Vert_{\X_{\eta^{(m)},\nu}} 
\end{equation}
can be evaluated on a computer using interval arithmetic.

\begin{remark}
This bound requires the evaluation of the parameterization $p$, since it involves $F^{(0)}_0(\bar \bX)$ which is defined as
\begin{equation*}
F^{(0)}_0(\bar \bX)= \bX^{(m)}_0+2\sum_{k=1}^\infty (-1)^k\bX^{(m)}_k - p(1).
\end{equation*} 
In practice, we only have an approximate parameterization $\hat p$, together with a validated error bound $\br$, see Section~\ref{sec:summary_manifold}. Thus, whenever we have to evaluate $p(1)$ we use
\begin{equation*}
p^{(i)}(1)\in \hat p^{(i)}(1) + \left[-\frac{\br}{\eta^{(i)}},\frac{\br}{\eta^{(i)}}\right]\quad \forall~1\leq i\leq 6.
\end{equation*}
\end{remark}

\subsubsection{The bound $Z_0$}

\begin{proposition}
\label{prop:Z0_orbit}
Let $\nu>1$ and $\eta\in\R^{6M}_{>0}$. Consider $\A^\dag$ and $\A$ as in~\eqref{def:Adag} and~\eqref{def:A}, and define $\B=I_{\X_{\eta,\nu}}-\A\A^\dag$. Then, for all $1\leq m\leq M$
\begin{align}
\label{def:Z0_orbit}
Z_0^{(m)} &= \sum_{n=1}^M \left\vert \left\Vert \B^{(m;n)}\right\Vert_\nu \right\vert_{\eta^{(n)}\to\eta^{(m)}} \nonumber \\
&= \sum_{n=1}^M\max_{1\leq j\leq 6} \frac{1}{\eta^{(n,j)}} \sum_{i=1}^6 \left\Vert \B^{(m,i;n,j)} \right\Vert_\nu \eta^{(m,i)} \nonumber \\
&= \sum_{n=1}^M\max_{1\leq j\leq 6} \frac{1}{\eta^{(n,j)}} \sum_{i=1}^6 \max_{0\leq l\leq K-1} \frac{1}{\xi_l(\nu)}\sum_{k=0}^{K-1} \left\vert \B^{(m,i;n,j)}_{k,l} \right\vert \xi_k(\nu) \eta^{(m,i)}
\end{align}
satisfies
\begin{equation*}
\left\Vert \left(I_{\X_{\eta,\nu}}-\A\A^\dag\right)^{(m)}\right\Vert_{\X_{\eta,\nu}\to\X_{\eta^{(m)},\nu}} \leq Z_0^{(m)}.
\end{equation*}
\end{proposition}
\begin{proof}
The result follows immediately from~\eqref{eq:ell_1_nu_op_norm} and~\eqref{eq:lin_op_norm_eta_nu}.
\end{proof}

We point out that, by construction of $\A^\dag$ and $\A$, each block $\B^{(m,i;n,j)}$ is finite (only the first $K\times K$ coefficients are potentially non zero), and hence the $\ell^1_\nu$ operator norms $\left\Vert \B^{(m,i;n,j)}\right\Vert_\nu$ can all be evaluated on a computer.

\subsubsection{The bound $Z_1$}
\label{sec:Z1_orbit}

\begin{proposition}
\label{prop:Z1_orbit}
Let $\nu>1$, $\eta\in\R^{6M}_{>0}$ and $\bar \bX\in\R^{6MK}$ (again identified with its injection in $\X_{\eta,\nu}$). Consider $\F$, $\A^\dag$ and $\A$ as in~\eqref{def:F_orbit},~\eqref{def:Adag} and~\eqref{def:A}. Define $\bC=\A(\txtD \F(\bar\bX)-\A^\dag)$ and 
\begin{equation*}
\fC^{(m,i;n,j)}_{{\textnormal{finite}}}=\max_{0\leq l\leq 4K-2} \frac{1}{\xi_l(\nu)} \sum_{k\in\N} \left\vert \bC^{(m,i;n,j)}_{k,l} \right\vert \xi_k(\nu).
\end{equation*}
Let $\bu^{(n;l,j)}$ be the vector of Chebyshev coefficients representing the partial derivative $\txtD _{(n,j)}f^{(n,l)}(\bar X)$ and define
\begin{equation*}
\left(\bu^{(n;l,j)}_{\mp}\right)_k =\left\{\begin{aligned}
& 0 \quad & k=0 \\
& \bu^{(n;l,j)}_{k-1}-\bu^{(n;l,j)}_{k+1}\quad & k\geq 1, \end{aligned}\right.
\end{equation*}
and
\begin{equation*}
\fC^{(m,i;n,j)}_{{\textnormal{tail}}}=\frac{2}{\nu^K}\left(\left\Vert \A^{(m,i;n,j)}_{\cdot,0}\right\Vert_\nu + \left\Vert \A^{(m,i;n+1,j)}_{\cdot,0}\right\Vert_\nu\right) +  \frac{\delta_{m,n}}{K}\frac{t^{(m)}-t^{(m-1)}}{4} \left\Vert \bu^{(n;i,j)}_{\mp}\right\Vert_\nu,
\end{equation*}
with the convention that $\left\Vert \A^{(m,i;n+1,j)}_{\cdot,s}\right\Vert_\nu=0$ if $n=M$. Finally, define 
\begin{equation*}
\fC^{(m,i;n,j)} = \max\left(\fC^{(m,i;n,j)}_{{\textnormal{finite}}},\fC^{(m,i;n,j)}_{{\textnormal{tail}}}\right)
\end{equation*}
and
\begin{align}
\label{def:Z1_orbit}
Z_1^{(m)} &= \sum_{n=1}^M \left\vert \fC^{(m;n)}\right\vert_{\eta^{(n)}\to\eta^{(m)}} \nonumber \\
&= \sum_{n=1}^M \max_{1\leq j\leq 6} \frac{1}{\eta^{(n,j)}} \sum_{i=1}^6 \fC^{(m,i;n,j)} \eta^{(m,i)}.
\end{align}
Then, for all $1\leq m\leq M$
\begin{equation*}
\left\Vert \left(\A(\txtD \F(\bar\bX)-\A^\dag)\right)^{(m)}\right\Vert_{\X_{\eta,\nu}\to\X_{\eta^{(m)},\nu}} \leq Z_1^{(m)}.
\end{equation*}
\end{proposition}
\begin{proof}
We start from 
\begin{align*}
\left\Vert \left(\A(\txtD \F(\bar\bX)-\A^\dag)\right)^{(m)}\right\Vert_{\X_{\eta,\nu}\to\X_{\eta^{(m)},\nu}} &= \left\Vert \bC^{(m)} \right\Vert_{\X_{\eta,\nu}\to\X_{\eta^{(m)},\nu}} \\
&\leq \sum_{n=1}^M \left\Vert \bC^{(m;n)} \right\Vert_{\X_{\eta^{(n)},\nu}\to\X_{\eta^{(m)},\nu}}  \\
& \leq \sum_{n=1}^M \max_{1\leq j\leq 6} \frac{1}{\eta^{(n,j)}} \sum_{i=1}^6 \left\Vert \bC^{(m,i;n,j)} \right\Vert_\nu \eta^{(m,i)}
\end{align*}
and estimate each $\ell^1_\nu$ operator norm $\left\Vert \bC^{(m,i;n,j)} \right\Vert_\nu$. We do so by computing explicitly (with a computer and interval arithmetic) the norm of a finite number of columns of $\bC^{(m,i;n,j)}$, and by estimating by hand the norm of the remaining columns, as emphasized by the following splitting: 
\begin{align*}
\left\Vert \bC^{(m,i;n,j)} \right\Vert_\nu &= \sup_{l\in\N} \frac{1}{\xi_l(\nu)} \sum_{k\in\N} \left\vert \bC^{(m,i;n,j)}_{k,l} \right\vert \xi_k(\nu) \\
&= \max \left[\max_{0\leq l\leq 4K-2} \frac{1}{\xi_l(\nu)} \sum_{k\in\N} \left\vert \bC^{(m,i;n,j)}_{k,l} \right\vert \xi_k(\nu),  \sup_{l>4K-2} \frac{1}{\xi_l(\nu)} \sum_{k\in\N} \left\vert \bC^{(m,i;n,j)}_{k,l} \right\vert \xi_k(\nu) \right] \\
&= \max \left[\fC^{(m,i;n,j)}_{{\textnormal{finite}}},  \sup_{l>4K-2} \frac{1}{\xi_l(\nu)} \sum_{k\in\N} \left\vert \bC^{(m,i;n,j)}_{k,l} \right\vert \xi_k(\nu) \right].
\end{align*} 
By definition of $\F$, the only non-zero blocks in $\txtD \F(\bar\bX)$ are the blocks $(\txtD \F(\bar\bX)-\A^\dag)^{(m;n)}$ for $m=n$ and $m=n+1$. Therefore we have
\begin{align}
\label{eq:C_deomposed}
 \bC^{(m,i;n,j)} &= \sum_{l=1}^6 \A^{(m,i;n,l)}(\txtD \F(\bar\bX)-\A^\dag)^{(n,l;n,j)} \nonumber\\
 &\quad + \sum_{l=1}^6 \A^{(m,i;n+1,l)}(\txtD \F(\bar\bX)-\A^\dag)^{(n+1,l;n,j)},
\end{align}
the second term being $0$ if $n=M$. This yields
\begin{align}
\label{eq:estimate_Z1_tail}
\sup_{l>4K-2} \frac{1}{\xi_l(\nu)} \sum_{k\in\N} \left\vert \bC^{(m,i;n,j)}_{k,l} \right\vert \xi_k(\nu) \leq \fC^{(m,i;n,j)}_{{\textnormal{tail}}}.
\end{align}

Estimate~\eqref{eq:estimate_Z1_tail} comes from a meticulous but rather straightforward analysis of the various terms appearing in~\eqref{eq:C_deomposed} for the \emph{columns} of indices $l> 4K-2$, using that
\begin{itemize}
\item $\A^{(m,i;n,l)}_{k,s} = \frac{1}{k}\delta_{k,s}\delta_{i,l}\delta_{m,n}$ whenever $s\geq K$;
\item $(\txtD \F(\bar\bX)-\A^\dag)^{(n,l;n,j)}_{k,s}=0$ for all $s>4K-2$ and $1\leq k\leq K-1$;
\item $\left\vert (\txtD \F(\bar\bX)-\A^\dag)^{(n,j;n,j)}_{0,s}\right\vert = 1$ for all $s>4K-2$ and $\left\vert (\txtD \F(\bar\bX)-\A^\dag)^{(n+1,j;n,j)}_{0,s}\right\vert = 1$ for all $s>4K-2$ if $n<M$;
\item $\left\vert (\txtD \F(\bar\bX)-\A^\dag)^{(n,l;n,j)}_{k,s}\right\vert = \left\vert \frac{t^{(n)}-t^{(n-1)}}{4} \left(\bu_{\mp}^{(n;l,j)}\right)_{\vert k-s\vert} \right\vert $ for all $s>4K-2$, $k\geq 1$.
\end{itemize}
Estimates similar to~\eqref{eq:estimate_Z1_tail} were obtained previously in~\cite{BerDesLesMir15,BerShe15}.
\end{proof}

Again, we point out that each \emph{row} and \emph{column} of $\A$ and $\txtD \F(\bar\bX)-\A^\dag$ only have a finite number of non-zero coefficients, hence the quantities $\fC^{(m,i;n,j)}_{{\textnormal{finite}}}$ and $\fC^{(m,i;n,j)}_{{\textnormal{tail}}}$ involved in the definition of $Z_1$ can indeed be evaluated with a computer.

\begin{remark}
\label{rem:eta_orbit}
In practice, the larger term in $Z_1^{(m)}$ comes from
\begin{equation*}
\left\vert \fC^{(m;m)}\right\vert_{\eta^{(m)}}=\max_{1\leq j\leq 6} \frac{1}{\eta^{(m,j)}} \sum_{i=1}^6 \fC^{(m,i;m,j)} \eta^{(m,i)},
\end{equation*} 
and thus we choose $\eta^{(m)}$ so as to minimize this quantity, as explained in Proposition~\ref{prop:frob}.
\end{remark}

\subsubsection{The bounds $Z_k$, $2\leq k\leq 4$}

\begin{proposition}
\label{prop:Z2_orbit}
Let $\nu>1$, $\eta\in\R^{6M}_{>0}$ and $\bar \bX\in\R^{6MK}$ (again identified with its injection in $\X_{\eta,\nu}$). Consider $\F$, $\A^\dag$ and $\A$ as in~\eqref{def:F_orbit},~\eqref{def:Adag} and~\eqref{def:A}, and define for all $2\leq k\leq 4$ and all $1\leq m\leq M$
\begin{equation}
\label{def:Zk_orbit}
Z_k^{(m)}=\sum_{n=1}^M \max_{1\leq j_1,\ldots,j_k \leq 6} \sum_{i_1=1}^6 \frac{\eta^{(m,i_1)}}{\eta^{(n,j_1)}\ldots\eta^{(n,j_k)}}\sum_{i_2=1}^6 \left\Vert \A^{(m,i_1;n,i_2)}\right\Vert_\nu \left\Vert \txtD ^k_{(n,j_1;\ldots;n,j_k)}\F^{(n,i_2)}(\bar\bX)\right\Vert_\nu.
\end{equation}
Then, for all $2\leq k\leq 4$ and all $1\leq m\leq M$
\begin{equation*}
\left\Vert \left(\A \txtD ^k\F(\bar\bX))\right)^{(m)}\right\Vert_{\X_{\eta,\nu}\to\X_{\eta^{(m)},\nu}} \leq Z_k^{(m)}.
\end{equation*}
\end{proposition}
\begin{proof}
This estimate follows readily from writing down what the norm of the multilinear operator $\left(\A \txtD ^k\F(\bar\bX))\right)^{(m)}$ is and simply using triangle inequalities, which is a straightforward computation albeit being quite lengthy due to the intricate norm used on $\X_{\eta,\nu}$.
\end{proof}
Again, each of the involved $\ell^1_\nu$ operator norms can be evaluated explicitly and rigorously on a computer, using interval arithmetic.

\begin{remark}
A slightly less sharp bound, but faster to evaluate on a computer, is given by
\begin{align*}
&\left\Vert \left(\A \txtD ^k\F(\bar\bX)\right)^{(m)}\right\Vert_{\left(\X_{\eta,\nu}\right)^k\to \X_{\eta^{(m)},\nu}}  \\
&\leq \sum_{n=1}^M \max_{1\leq j_1,\ldots,j_k \leq 6} \sum_{i_2=1}^6 \sum_{i_1=1}^6 \frac{\eta^{(m,i_1)}}{\eta^{(n,i_2)}} \left\Vert \A^{(m,i_1;n,i_2)}\right\Vert_\nu \frac{\eta^{(n,i_2)}}{\eta^{(n,j_1)}\ldots\eta^{(n,j_k)}}\left\Vert \txtD ^k_{(n,j_1;\ldots;n,j_k)}\F^{(n,i_2)}(\bar\bX)\right\Vert_\nu  \\
&\leq \sum_{n=1}^M \left(\max_{1\leq i_2\leq 6}\sum_{i_1=1}^6 \frac{\eta^{(m,i_1)}}{\eta^{(n,i_2)}} \left\Vert \A^{(m,i_1;n,i_2)}\right\Vert_\nu\right)\left(\max_{1\leq j_1,\ldots,j_k \leq 6} \sum_{i_2=1}^6  \frac{\eta^{(n,i_2)}}{\eta^{(n,j_1)}\ldots\eta^{(n,j_k)}}\left\Vert \txtD ^k_{(n,j_1;\ldots;n,j_k)}\F^{(n,i_2)}(\bar\bX)\right\Vert_\nu\right)  \\
&= \sum_{n=1}^M \left\vert \left\Vert \A^{(m;n)}\right\Vert_{\nu}\right\vert_{\eta^{(n)}\to\eta^{(m)}} \left\vert \left\Vert \txtD ^k\F^{(n)}(\bar\bX)\right\Vert_\nu \right\vert_{\left(\eta^{(n)}\right)^k\to\eta^{(n)}}.
\end{align*}
\end{remark}

\subsection{Summary}

We recall that we explained in Section~\ref{sec:step_zeros} how to obtain a rigorous enclosure of minima and saddle points of $V$, in Section~\ref{sec:trapping} how to validate trapping regions around minima, and in Sections~\ref{sec:step_manifold} and~\ref{sec:manifold} how to validate local unstable manifold for the saddles. We now show how all the estimates derived in this section can be combined with Corollary~\ref{cor:F} to rigorously validate a connecting orbit between a saddle and a minimum.

\begin{theorem}
\label{th:recap_orbit}
Let $\hat p:[-1,1]\to \R^6$ be a validated parameterization of a local unstable manifold of a saddle, with a validation radius $\br$, as provided by Theorem~\ref{th:recap_manifold}. Let $(x_1,y_1)\in \R^2$ be a minimum of $V$, and $U\subset \R^2$ be a validated trapping region around $(x_1,y_1)$. Let $\nu>1$, $M\in\N_{\geq 1}$, $K\in\N_{\geq 1}$, $\tau>0$ and $\eta\in\R^{6M}_{>0}$. 

Let $\bar \bX = (\bar\bX_0,\ldots,\bar\bX_{K-1})\in \R^{6MK}$ and consider $\F$, $\A^\dag$ and $\A$ as defined in Section~\ref{sec:framework_orbit}. Consider also the bounds $Y$, $Z_0$, $Z_1$ and $Z_k$, $2\leq k\leq 4$ defined in~\eqref{def:Y_orbit}, \eqref{def:Z0_orbit}, \eqref{def:Z1_orbit} and \eqref{def:Zk_orbit} respectively. Finally, assume there exists $\brho>0$ such that, for all $1\leq m\leq M$, $P^{(m)}(\brho)<0$ and $Q^{(m)}(\brho)<0$, with $P^{(m)}$ and $Q^{(m)}$ defined in~\eqref{def:radii_pol_F}.

Then there exists $\bX\in\X_{\eta,\nu}$ satisfying
\begin{equation}
\label{eq:err_X}
\left\Vert \bX -\bar\bX \right\Vert_{\X_{\eta,\nu}} \leq \brho 
\end{equation}
such that the function $X$ defined as in~\eqref{eq:piecewise_cheb} satisfies $X'=f(X)$ on $[0,\tau]$ and such that $X(0)$ belongs to the unstable manifold of the saddle. Besides, if $(X^{(1)}(\tau),X^{(2)}(\tau))\in U$, this yields that the existence of a connecting orbit from the saddle to the minimum for the vector field $-\nabla V$.
\end{theorem}

Notice that~\eqref{eq:err_X} gives a control of the error in phase space, namely
\begin{equation*}
\left\vert X^{(i)}(t) -\bar X^{(i)}(t) \right\vert \leq \frac{\brho}{\eta^{(m,i)}},\quad \forall~t^{(m-1)} \leq t \leq t^{(m)}.
\end{equation*}

We give in Section~\ref{sec:results} several examples (with explicit values of the parameters) of applications of Theorem~\ref{th:recap_orbit} to validate a MEP for the M\"uller-Brown potential.

\section{Conclusions}
\label{sec:conclusions}

\subsection{Results for the M\"uller-Brown potential}
\label{sec:results}

We use Theorem~\ref{th:recap_manifold} and Theorem~\ref{th:recap_orbit} to validate a MEP for the M\"uller-Brown potential. This MEP connects two minima through two saddles and one extra minimum (see Figure~\ref{fig:zeros}), and is thus composed of four connecting orbits:
\begin{center}
Min$_1$ $\longleftarrow$ Sad$_1$ $\longrightarrow$ Min$_2$ $\longleftarrow$ Sad$_2$ $\longrightarrow$ Min$_3$.
\end{center}

\begin{figure}[htbp]
\centering
\includegraphics[width=\linewidth]{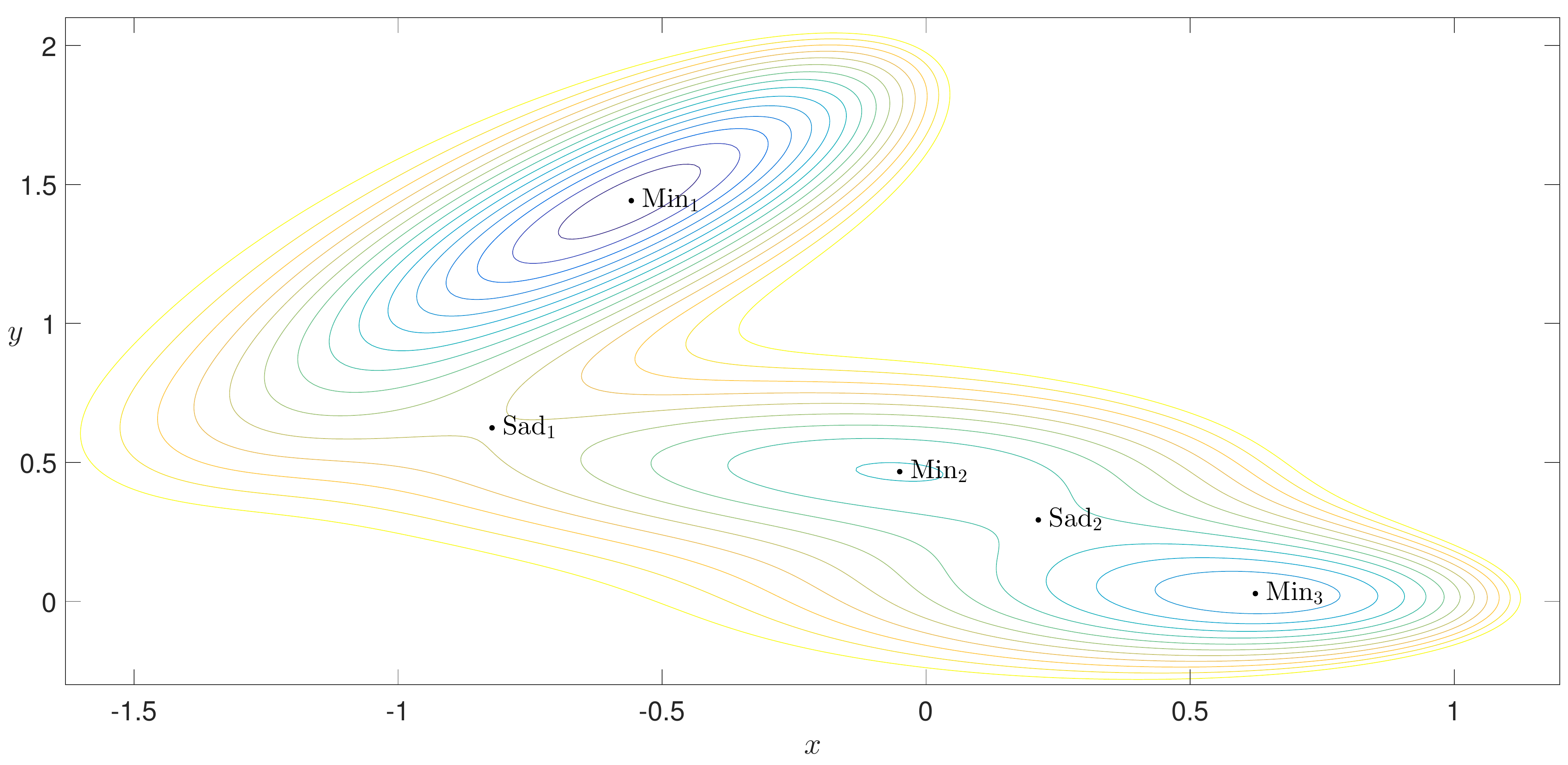}
\caption{The minima and saddles that \emph{organize} the MEP for the M\"uller-Brown potential, represented in the $(x,y)$ phase space.}
\label{fig:zeros}
\end{figure}

To validate this MEP, we follow the main steps described in Section~\ref{sec:main_steps}. We now make precise for each step the values of the various parameters that we use and present the output of the whole procedure.

\paragraph{(I) Locate and validate saddles and minima.}
Using the method exposed in Section~\ref{sec:step_zeros}, we obtain the following validated enclosure for the saddles and minima.
\begin{equation*}
\text{Min}_1=  \begin{pmatrix}
0.558223634633024 \\
1.441725841804669
\end{pmatrix} 
\pm 2.3\times 10^{-16}, \quad
\text{Min}_2=  \begin{pmatrix}
-0.050010822998206 \\
0.466694104871972
\end{pmatrix} 
\pm 3.1\times 10^{-16},
\end{equation*}
\begin{equation*}
\text{Min}_3=  \begin{pmatrix}
0.623499404930877 \\
0.028037758528686
\end{pmatrix} 
\pm 6.7\times 10^{-16},
\end{equation*}
\begin{equation*}
\text{Sad}_1=  \begin{pmatrix}
-0.822001558732732 \\
0.624312802814871
\end{pmatrix} 
\pm 6.7\times 10^{-16}, \quad
\text{Sad}_2=  \begin{pmatrix}
0.212486582000662 \\
0.292988325107368
\end{pmatrix} 
\pm 6.4\times 10^{-16}.
\end{equation*}

\paragraph{(II) Introduce an equivalent polynomial vector field.} 
This step was already fully detailed in Section~\ref{sec:polynomial_reformulation}.

\paragraph{(IIIa) Compute and validate the local unstable manifold of the saddles.}
We use here the method outlined in Section~\ref{sec:step_manifold} and detailed in Section~\ref{sec:manifold}. First, we obtain a validated enclosure for the unstable eigenvalue of each saddle, and for an associated eigenvector (in the extended phase space):
\begin{equation*}
\lambda_1=750.8626628392770 \pm 2.2\times 10^{-10}, \quad \lambda_2=735.2472621113654 \pm 2.3\times 10^{-10},
\end{equation*}
and
\begin{equation*}
v_1= \begin{pmatrix}
-0.004365095236381 \\
   0.003716635857366 \\
   0.009144362933730 \\
   0.715690888528226 \\
  -0.697810674544410 \\
  -0.027024576917547 
\end{pmatrix} \pm 1.6\times 10^{-12},\quad 
v_2=\begin{pmatrix}
-0.001396858271947 \\
   0.002417454704800 \\
   0.746139634477447 \\
  -0.660215837640495 \\
  -0.000000003332255 \\
  -0.085923793504697
\end{pmatrix} \pm 1.1\times 10^{-12}.
\end{equation*}
Then, we fix a scaling for the eigenvector, namely $\gamma_1=5$ for the first one and $\gamma_2=15$, and for each saddle point solve recursively the invariance equation~\eqref{eq:invariance_coeff} for a finite number of coefficients. Using interval arithmetic, we compute $N_1=20$ (resp. $N_2=30$) coefficients for the parameterization of the unstable manifold of Sad$_1$ (resp. Sad$_2$). The values of the coefficients can be found in the file \texttt{coeff\_para}. 

\begin{remark}
We recall that $\gamma_1$ and $\gamma_2$ act as scaling parameters for the parameterization coefficients $\p_1=(\p_1)_n$ and $\p_2=(\p_2)_n$. Their value were chosen so as to get a reasonable decay of the coefficients. Indeed, we want the last few computed coefficients to be of an order close to machine precision, so that the validation (i.e. proving that the missing tail is close to $0$ by applying Theorem~\ref{th:recap_manifold}) is likely to succeed, but we don't want to many of them too be to small, because then they don't contain much information on the manifold. For more details on how this choice can be done in an optimized way (also with regard to the computational cost), see~\cite{BreLesMir16}.
\end{remark}

Finally, we use Theorem~\ref{th:recap_manifold} to validate these two truncated parameterization. As explained in Remark~\ref{rem:eta_manifold}, we use Proposition~\ref{prop:frob} to find weights $\eta\in\R^6_{>0}$ giving an optimal $Z_1$ bound:
\begin{equation*}
\eta_1=\begin{pmatrix}
0.233475274346582 \\
   0.972240260296501 \\
   0.012217320748587 \\
   0.002816090862926 \\
   0.008991277632639 \\
   0.000368998496530
\end{pmatrix},\quad 
\eta_2=\begin{pmatrix}
0.013610077886795 \\
   0.999830877794874 \\
   0.002743858348424 \\
   0.001900890855733 \\
   0.011908138216216 \\
   0.000188911061966
\end{pmatrix}.
\end{equation*}

With these parameters we manage to validate both manifolds, that is we find validation radiuses $\br_1=1.024225153462953\times 10^{-17}$ and $\br_2=1.975153579406591\times 10^{-18} $ for which the radii polynomials $P$ and $Q$ are negative. See Figure~\ref{fig:with_orbit} for a representation of the validated manifolds in (the original two-dimensional) phase space.

\paragraph{(IIIb) Validate trapping regions around the minima.}
This step is more straightforward. Using the procedure described in Section~\ref{sec:trapping}, we prove using \textsc{Intlab} that for each of the three minimum, the square of length $0.02$ centered at the minimum is included in the stable manifold of the minimum. See Figure~\ref{fig:with_orbit} for a representation of this trapping regions in phase space.

\paragraph{(IIIc) Compute and validate orbits connecting unstable manifolds and trapping regions.}
First, we numerically compute orbits starting from the end of the validated unstable manifold, to find after which integration time the orbit enters the trapping region validated at the previous step. This yields $\tau_{11}=0.015$, $\tau_{12}=0.025$, $\tau_{22}=0.012$ and $\tau_{23}=0.018$. In all this paragraph, indices $_{ij}$ are used to denote a quantity related to the orbit from Sad$_i$ to Min$_j$. Then, we compute a piece-wise Chebyshev series representation of these approximate solutions. We use a uniform subdivision of size $M_{11}=M_{12}=M_{22}=M_{23}=10$, and on each subdivision use $K_{11}=30$ (resp. $K_{12}=K_{22}=K_{23}=20$) Chebyshev coefficients. The values of the coefficients can be found in the file \texttt{coeff\_orbit}. Then, we use Theorem~\ref{th:recap_orbit} to validate these four approximate orbits. As explained in Remark~\ref{rem:eta_orbit}, we use Proposition~\ref{prop:frob} to find weights $\eta\in\R^{6M}_{>0}$ giving an optimal $Z_1$ bound. The exact value of these weights can be found in the file \texttt{weights\_orbit}. With the weights $\nu_{11}=1.4$, $\nu_{12}=1.5$, $\nu_{22}=1.7$ and $\nu_{23}=1.5$ for the $\ell^1_\nu$ norms, we manage to validate the four orbits, that is we find validation radiuses $\brho_{11}=8.568033048134257\times 10^{-11}$, $\brho_{12}=1.493891911294753\times 10^{-11}$, $\brho_{22}=5.071535766518771\times 10^{-10}$ and $\brho_{11}=6.190811910154036\times 10^{-10}$ for which the radii polynomials $P^{(m)}$ and $Q^{(m)}$ are all negative. See Figure~\ref{fig:with_orbit} for a representation of the validated orbits in (the original two-dimensional) phase space. Finally, as suggested by Figure~\ref{fig:with_orbit}, we check using interval arithmetic and taking into account the validation radiuses $\rho_{ij}$ that each orbit indeed ends in the validated trapping region, which concludes the validation of each connecting orbit, and thus of the MEP. All the validation estimates can be obtained by running the script \texttt{script\_main} available at~\cite{BreKue18_code}, where the files \texttt{coeff\_para}, \texttt{coeff\_orbit} and \texttt{weights\_orbit} can also be found.

\begin{figure}[htbp]
\centering
\includegraphics[width=\linewidth]{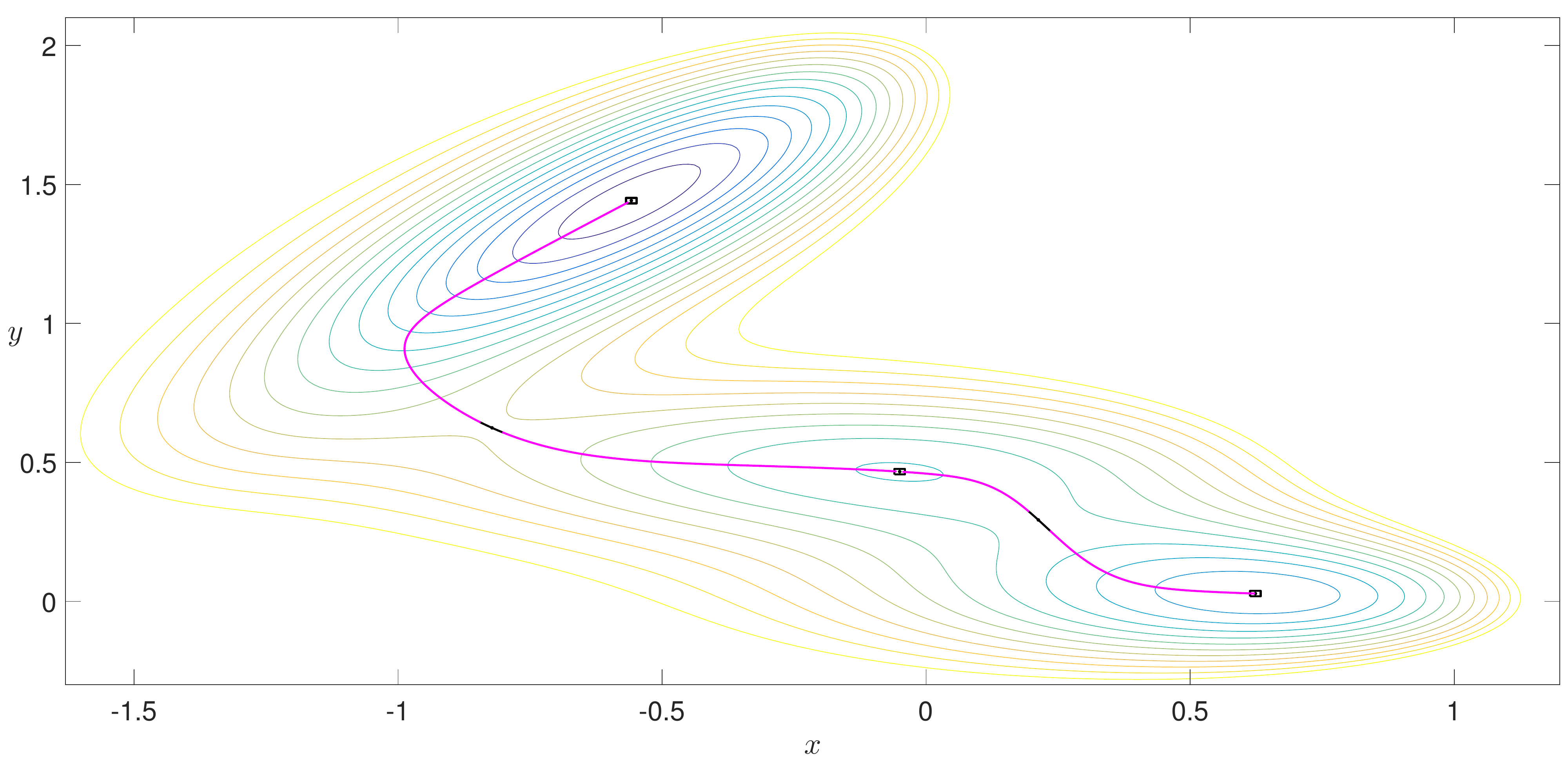}
\caption{The validated MEP represented in the $(x,y)$ phase space. The black parts of the orbits near each saddle represent the validated unstable manifolds, whereas the magenta parts are the one validated using piece-wise Chebyshev series. The small boxes around each minimum are the validated trapping regions.}
\label{fig:with_orbit}
\end{figure}

\subsection{Possible generalizations}
\label{sec:generalizations}

To make the exposition as clear as possible, most of the technical estimates presented were tailored for the M\"uller-Brown example. We end this paper by discussing to what extent the different parts of the method described can be generalized.

Our first comment concerns the gradient structure, and we explain here how to generalize our approach to non-gradient systems. As the reader may have noticed, nowhere in the a posteriori validation process for heteroclinic orbits did we use that the original problem had a gradient structure (in fact we even lost that feature through the polynomial reformulation). However, the gradient structure is crucial to ensure that heteroclinic paths of~\eqref{eq:ODE} actually correspond to MEP of~\eqref{eq:SDE}. For general SODEs of the form~\eqref{eq:SDE1} where $f$ is not a gradient, this may no longer be true, especially for MAPs going upward~\cite{HeyVan08bis} (i.e. from a minimum to a saddle). To solve this issue, one can use a different approach to solve the minimization problem~\eqref{eq:minimization}, namely look at Hamilton's equations of motion associated to the Lagrangian~\eqref{eq:FW1} (see e.g.~\cite{GraSchVan17}). For this extended system heteroclinic paths again correspond to MAPs, even in the non-gradient case, and the whole machinery presented in Sections~\ref{sec:main_steps} to~\ref{sec:connection} is then applicable.

Our second comment concerns the different behaviors that can result of having saddle points of Morse index greater than one (see for instance~\cite{CamKohVan11}). For systems of dimension three or more, besides \emph{saddle to minimum} connections a MEP could contain \emph{saddle to saddle} connections. Our approach generalizes with only small modifications to such cases. Indeed, using the method described in Section~\ref{sec:main_steps} one can compute and validate a local stable manifold of the saddle where the orbit arrives. This stable manifold then plays the role of the trapping region of the minimum in Section~\ref{sec:connection}, and one can validate the saddle to saddle heteroclinic orbit by rigorously solving the boundary value problem between the stable and the unstable manifold (we again refer to~\cite{Mir17} for more details).  As mentioned in~\cite{CamKohVan11}, saddle points of Morse index greater than one can also lead to situations where MEPs are not unique (not even locally), because some heteroclinic connections come as a continuous families. Geometrically this corresponds to situations where the intersection between the unstable manifold of a saddle and the stable manifold of a minimum (or a saddle) is more than one-dimensional. It is not completely clear to us what one would actually want to compute in such situation, but from the a posteriori validation point of view there would be at least two available options. The first one would be to recover local uniqueness by focusing on a specific orbit among the family, for instance by only computing and validating submanifolds of the stable/unstable manifolds corresponding to \emph{fast} directions. Another option would be to parametrize such family of heteroclinic orbits (for instance by the \emph{exit point} on the unstable manifold) and use a parameter-dependent version of the Banach fixed point theorem to validate the whole family of orbits (see for instance~\cite{BreLesVan13,GamLesPug16,Les18,Plu95} and the references therein).

Our last comment concerns the practical applicability (mainly from a computational cost point of view) of our approach to high dimensional systems. First, let us mention that, to the best of our knowledge, there has never been an attempt to apply this kind of a posteriori validation techniques to high dimensional system of ODEs. Nonetheless, such techniques were successfully used to study a different kind of high dimensional systems, namely 3D PDEs~\cite{BerWil18}. Therefore, we believe that, with some implementation effort, our approach could be adapted to handle higher dimensional systems. However, there is one aspect of the method described in this paper that is clearly ill-behaved with respect to the dimension: the polynomial reformulation. As explained in Section~\ref{sec:polynomial_reformulation}, this reformulation allowed us to simplify the presentation (and the implementation) but at the cost of an increase in the dimension. While this trade-off is acceptable for low dimensional examples, it limits the applicability for higher dimensional ones. An alternative approach could be to keep the initial (non-polynomial) nonlinearities, and extend our method (both from the estimates and from the implementation side) to handle these  terms. Such an extension would require substantial, but probably within reach, modifications. Indeed rigorous numerics techniques similar to those used in this work were already successfully applied directly to systems with non-polynomial linearities (see e.g.~\cite{BreHorMcKPlu06,DayKal13,Yam98}).

\paragraph*{Acknowledgments:} MB and CK have been supported by a Lichtenberg Professorship of the VolkswagenStiftung. We also thank two referees for their helpful comments improving the presentation of our results.


\end{document}